\documentclass[11pt]{article}
\usepackage[center]{caption}
\usepackage{amsthm,amsmath,amsfonts,amssymb}
\usepackage{hyperref}
\usepackage{fullpage}
\usepackage{tikz,tkz-graph,tkz-berge}
\usepackage{verbatim}
\usetikzlibrary{arrows}
\usetikzlibrary{shapes}

\usepackage{subfig}

\newtheorem*{thmA}{Theorem A}
\newtheorem*{thmB}{Theorem B}
\newtheorem{conj}{Conjecture}
\newtheorem{lemma}{Lemma}
\newtheorem{cor}[lemma]{Corollary}
\newtheorem*{mainthm}{Main Theorem}
\newtheorem*{atmainthm}{Main Theorem for $\AT$}

\theoremstyle{definition}

\def\mv{\overrightarrow}

\def\adj{\leftrightarrow}
\def\nonadj{\not\leftrightarrow}
\def\join{\vee}
\def\diff{{\rm diff}}

\def\Ceil#1{\left\lceil #1 \right\rceil}
\def\C{{\mathcal{C}}}
\def\D{{\mathcal{D}}}
\def\F{{\mathcal{F}}}
\def\G{{\mathcal{G}}}

\def\Z{{\mathbb{Z}}}
\def\chil{\chi_{\ell}}
\def\chiol{\chi_{p}}

\newcommand{\AT}{\operatorname{AT}}

\author{Daniel W. Cranston\thanks{Department of Mathematics and Applied
Mathematics, Viriginia Commonwealth University; \texttt{dcranston@vcu.edu}} 
\and Landon Rabern\thanks{LBD Software Solutions; \texttt{landon.rabern@gmail.com}}}
\title{Painting Squares in $\Delta^2-1$ Shades}
\begin{document}
\maketitle
\abstract{
\def\chil{\chi_{\ell}}
Cranston and Kim conjectured that if $G$ is a connected graph with maximum
degree $\Delta$ and $G$ is not a Moore Graph, then $\chil(G^2)\le \Delta^2-1$;
here $\chil$ is the list chromatic number.  We prove their conjecture; in
fact, we show that this upper bound holds even for online list chromatic number.
\\

}
\section{Introduction}

Graph coloring has a long history of upper bounds on a graph's
chromatic number $\chi$ in terms of its maximum degree $\Delta$.  A greedy
coloring (in any order) gives the trivial upper bound $\chi\le \Delta+1$.
In 1941, Brooks \cite{Brooks41} proved the following strengthening: If $G$ is
a graph with maximum degree $\Delta\ge 3$ and clique number $\omega \le
\Delta$, then $\chi\le \Delta$.  In 1977, Borodin and Kostochka
\cite{BorodinK77} conjectured the following further strengthening.

\begin{conj}[Borodin-Kostochka Conjecture~\cite{BorodinK77}]
If $G$ is a graph with $\Delta\ge 9$ and $\omega\le \Delta-1$, then $\chi\le
\Delta-1$.
\label{BKconj}
\end{conj}

\begin{figure}[h!tb]
\begin{center}
\includegraphics[scale=.5]{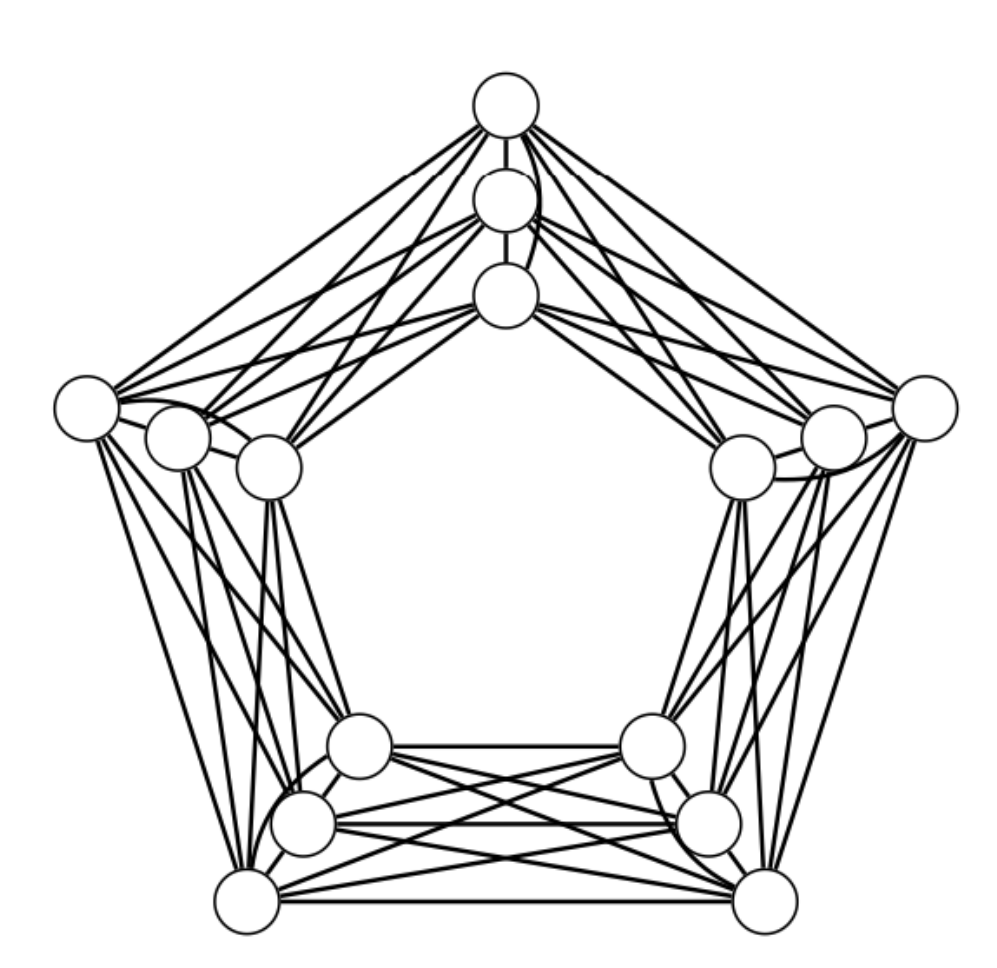}
\end{center}
\caption{The hypothesis $\Delta\ge 9$ in the
Borodin--Kostochka Conjecture is best possible.}
\label{BK-pic}
\end{figure}

If true, this conjecture is best possible in two senses.  First, the condition
$\Delta\ge 9$ cannot be dropped (or even weakened), as shown by the following
graph (See Figure \ref{BK-pic}).  Let $D_i$ induce a triangle for each $i\in\{1,\ldots,5\}$; if
$|i-j|\equiv
1\pmod 5$, then add all edges between vertices of $D_i$ and $D_j$.  This yields an 8-regular graph on 15 vertices with clique
number 6 and chromatic
number 8; it would be a counterexample to the conjecture if we weakened the
hypothesis $\Delta\ge 9$.  Similarly, even if we require $\omega\le \Delta-2$,
we cannot conclude that $\chi\le \Delta-2$, as is show by the join of a 
clique and a 5-cycle.  
For each $\Delta\in\{3,\ldots,8\}$, examples are known
\cite{CranstonR13claw,KingLP12} where $\omega \le \Delta-1$ but
$\chi = \Delta$.  Kostochka has informed us that already in 1977 when he and
Borodin posed Conjecture \ref{BKconj}, they believed the following stronger
``list version'' was true; however they omitted this version from their paper,
and it appeared in print \cite{CranstonR13claw} only in 2013.  We define the
list
chromatic number, denoted $\chil$, 
in Section~2 below.

\begin{conj}[Borodin-Kostochka Conjecture (list version)]
If $G$ is a graph with $\Delta\ge 9$ and $\omega\le \Delta-1$, then $\chil\le
\Delta-1$.
\label{BKlistconj}
\end{conj}

The purpose of this paper is to prove the following conjecture of Cranston and
Kim~\cite{CranstonK08}. 
In fact, we will prove this conjecture in the more general setting
of online list coloring.  It is easy to show, as we do below, that
Conjecture~\ref{BKlistconj} implies Conjecture \ref{CKconj}.

\begin{conj}[Cranston-Kim \cite{CranstonK08}]
If $G$ is a connected graph with maximum degree $\Delta\ge 3$, and $G$ is not a
Moore graph, then $\chil(G^2)\le \Delta^2-1$.
\label{CKconj}
\end{conj}

A Moore graph is a $\Delta$-regular graph $G$ on
$\Delta^2+1$ vertices such that $G^2=K_{\Delta^2+1}$; the sole example when
$\Delta=3$ is the Petersen graph.  Hoffman and Singleton~\cite{HoffmanS60}
famously proved that Moore graphs exist only when $\Delta\in \{2,3,7,57\}$. 
When $\Delta\in\{2,3,7\}$ Moore graphs exist and are known to be unique, and
when $\Delta=57$ no Moore graph is known.

In 2008 Cranston and Kim \cite{CranstonK08} proved Conjecture~\ref{CKconj} when
$\Delta=3$, and suggested that a similar but more detailed approach might
prove the whole conjecture.  
As mentioned above, it is easy to show that
Conjecture~\ref{CKconj} is implied by Conjecture~\ref{BKlistconj}.
The key is the following easy lemma at the end of \cite{CranstonK08}: 
If $G$ is connected and is not a Moore graph and $G$ has maximum degree $\Delta
\ge3$, then $G^2$ has clique number at most $\Delta^2 - 1$.
The proof is short once we have a result of Erd\H{o}s, Fajtlowicz,
and Hoffman~\cite{ErdosFH80} stating that a ``near-Moore graph'', i.e., 
a $\Delta$-regular graph such that $G^2=K_{\Delta^2}$, exists only when
$\Delta=2$.  For details, see the start of the proof of the Main
Theorem.

We note that recently Conjecture \ref{CKconj} was generalized to
higher powers.  Let $M$ denote the maximum possible degree when a graph of
maximum degree $k$ is raised to the $d$th power, i.e., vertices are adjacent in
$G^d$ if they are distance at most $d$ in $G$.  Miao and Fan~\cite{MiaoF13+}
conjectured that if $G$ is connected and $G^d$ is not $K_{M+1}$, then we can
save one color over the bound given by Brooks Theorem, i.e., $\chi(G^d)\le M-1$.
This was proved by Bonamy and Bousquet~\cite{BonamyB13+} in the more general
context of online list coloring.

%
\smallskip

The following conjecture is due to Wegner~\cite{Wegner77}, in the late 1970's. 
It is a less well-known variant of Wegner's analogous conjecture when the class
$\G_k$ is restricted to planar graphs.

\begin{conj}[Wegner~\cite{Wegner77}]
For each fixed $k$, let $\G_{k}$ denote the class of all graphs with
maximum degree at most $k$ and form $\G^2_{k}$ by taking the square
$G^2$ of each graph $G$ in $\G_{k}$. 
Now $\max_{H\in
\G^2_{k}}\chi(H)=\max_{H\in \G^2_{k}}\omega(H)$.
\label{wegner}
\end{conj}

Wegner in fact posed a more general conjecture for all powers of $\G_{k}$;
however, here we restrict our attention to Conjecture~\ref{wegner},
specifically for small values of $k$.  For each $H\in \G^2_k$, we have
$\Delta(H)\le k^2$, so Brooks' Theorem implies that $\chi(H)\le k^2$ unless
some component of $H$ is $K_{k^2+1}$.  For $k=1$ Wegner's Conjecture is
trivial. For
$k\in \{2,3,7\}$ it is easy; in each case $\G_k$ contains a Moore graph $G$,
and letting $H=G^2$, we have $H=K_{k^2+1}$, so $\chi(H)=\omega(H)=k^2+1$.  Thus, the first two
open cases of Conjecture~\ref{wegner} are $k=4$ and $k=5$.  Our Main Theorem
shows that every graph $G$ in $\G_4$ satisfies $\chil(G^2)\le 15$ and every
graph $G$ in $\G_5$ satisfies $\chil(G^2)\le 24$.  Matching lower bounds are
shown in Figure~2: we have $G_1\in \G_4$ with $\omega(G_1^2)=15$ and $G_2\in
G_5$ with $\omega(G_2^2)=24$. Both graphs were discovered by
Elspas (\cite{Elspas64} and p.~14 of \cite{MillerS-Survey}) and are known to be
the unique graphs $G$ with $\Delta\in\{4,5\}$ and $G^2 = K_{\Delta^2-1}$. 
This confirms Wegner's Conjecture when $k=4$ and $k=5$.

\begin{figure}
\begin{center}
\begin{tikzpicture}[scale = 5]
\tikzstyle{VertexStyle}=[shape = circle,	
minimum size = 5pt, inner sep = 2pt,
                                 draw]
\Vertex[x = 1.5, y = -0.0499999523162842, L = \tiny {}]{v0}
\Vertex[x = 1.60000002384186, y = 0.100000023841858, L = \tiny {}]{v1}
\Vertex[x = 1.35000002384186, y = 0.199999988079071, L = \tiny {}]{v2}
\Vertex[x = 1.54999995231628, y = 0.199999988079071, L = \tiny {}]{v3}
\Vertex[x = 1.5, y = 0.449999988079071, L = \tiny {}]{v4}
\Vertex[x = 1.60000002384186, y = 0.300000011920929, L = \tiny {}]{v5}
\Vertex[x = 1.79999995231628, y = 0.449999988079071, L = \tiny {}]{v6}
\Vertex[x = 1.70000004768372, y = 0.300000011920929, L = \tiny {}]{v7}
\Vertex[x = 1.95000004768372, y = 0.199999988079071, L = \tiny {}]{v8}
\Vertex[x = 1.75, y = 0.199999988079071, L = \tiny {}]{v9}
\Vertex[x = 1.79999995231628, y = -0.0499999523162842, L = \tiny {}]{v10}
\Vertex[x = 1.70000004768372, y = 0.100000023841858, L = \tiny {}]{v11}
\Vertex[x = 2.25, y = -0.149999976158142, L = \tiny {}]{v12}
\Vertex[x = 1.04999995231628, y = -0.149999976158142, L = \tiny {}]{v13}
\Vertex[x = 1.64999997615814, y = 0.849999994039536, L = \tiny {}]{v14}
\Edge[](v0)(v1)
\Edge[](v0)(v3)
\Edge[](v0)(v11)
\Edge[](v0)(v12)
\Edge[](v1)(v2)
\Edge[](v1)(v7)
\Edge[](v1)(v10)
\Edge[](v2)(v3)
\Edge[](v2)(v5)
\Edge[](v2)(v14)
\Edge[](v3)(v4)
\Edge[](v3)(v9)
\Edge[](v4)(v5)
\Edge[](v4)(v7)
\Edge[](v4)(v13)
\Edge[](v5)(v6)
\Edge[](v5)(v11)
\Edge[](v6)(v7)
\Edge[](v6)(v9)
\Edge[](v6)(v12)
\Edge[](v7)(v8)
\Edge[](v8)(v9)
\Edge[](v8)(v11)
\Edge[](v8)(v14)
\Edge[](v9)(v10)
\Edge[](v10)(v11)
\Edge[](v10)(v13)
\Edge[](v12)(v13)
\Edge[](v12)(v14)
\Edge[](v13)(v14)
\end{tikzpicture}
\hspace{.5in}
\begin{tikzpicture}[scale = 5]
\tikzstyle{VertexStyle}=[shape = circle, minimum size = 5pt, inner sep = 2pt, draw]
\Vertex[x = 0.800000011920929, y = 0.75, L = \tiny {}]{v0}
\Vertex[x = 0.800000011920929, y = 0.149999976158142, L = \tiny {}]{v1}
\Vertex[x = 1.79999995231628, y = 0.75, L = \tiny {}]{v2}
\Vertex[x = 1.79999995231628, y = 0.149999976158142, L = \tiny {}]{v3}
\Vertex[x = 1.29999995231628, y = 0.149999976158142, L = \tiny {}]{v4}
\Vertex[x = 1.29999995231628, y = 0.75, L = \tiny {}]{v5}
\Vertex[x = 0.649999976158142, y = 0.399999976158142, L = \tiny {}]{v6}
\Vertex[x = 0.949999988079071, y = 0.399999976158142, L = \tiny {}]{v7}
\Vertex[x = 0.800000011920929, y = 0.399999976158142, L = \tiny {}]{v8}
\Vertex[x = 1.95000004768372, y = 0.399999976158142, L = \tiny {}]{v9}
\Vertex[x = 1.79999995231628, y = 0.399999976158142, L = \tiny {}]{v10}
\Vertex[x = 1.64999997615814, y = 0.399999976158142, L = \tiny {}]{v11}
\Vertex[x = 1.45000004768372, y = 0.399999976158142, L = \tiny {}]{v12}
\Vertex[x = 1.14999997615814, y = 0.399999976158142, L = \tiny {}]{v13}
\Vertex[x = 1.29999995231628, y = 0.399999976158142, L = \tiny {}]{v14}
\Vertex[x = 0.649999976158142, y = -0.25, L = \tiny {}]{v15}
\Vertex[x = 0.949999988079071, y = -0.25, L = \tiny {}]{v16}
\Vertex[x = 0.800000011920929, y = 0, L = \tiny {}]{v17}
\Vertex[x = 1.29999995231628, y = 0, L = \tiny {}]{v18}
\Vertex[x = 1.45000004768372, y = -0.25, L = \tiny {}]{v19}
\Vertex[x = 1.14999997615814, y = -0.25, L = \tiny {}]{v20}
\Vertex[x = 1.64999997615814, y = -0.25, L = \tiny {}]{v21}
\Vertex[x = 1.95000004768372, y = -0.25, L = \tiny {}]{v22}
\Vertex[x = 1.79999995231628, y = 0, L = \tiny {}]{v23}
\Edge[](v0)(v3)
\Edge[style = {bend left}](v0)(v5)
\Edge[style = {ultra thick}](v0)(v6)
\Edge[style = {ultra thick}](v0)(v7)
\Edge[style = {ultra thick}](v0)(v8)
\Edge[](v1)(v2)
\Edge[style = {bend right}](v1)(v4)
\Edge[style = {ultra thick}](v1)(v6)
\Edge[style = {ultra thick}](v1)(v7)
\Edge[style = {ultra thick}](v1)(v8)
\Edge[style = {bend right}](v2)(v5)
\Edge[style = {ultra thick}](v2)(v9)
\Edge[style = {ultra thick}](v2)(v10)
\Edge[style = {ultra thick}](v2)(v11)
\Edge[style = {bend left}](v3)(v4)
\Edge[style = {ultra thick}](v3)(v9)
\Edge[style = {ultra thick}](v3)(v10)
\Edge[style = {ultra thick}](v3)(v11)
\Edge[style = {ultra thick}](v4)(v12)
\Edge[style = {ultra thick}](v4)(v13)
\Edge[style = {ultra thick}](v4)(v14)
\Edge[style = {ultra thick}](v5)(v12)
\Edge[style = {ultra thick}](v5)(v13)
\Edge[style = {ultra thick}](v5)(v14)
\Edge[style = {ultra thick}](v6)(v15)
\Edge[](v6)(v19)
\Edge[](v6)(v22)
\Edge[style = {ultra thick}](v7)(v16)
\Edge[](v7)(v20)
\Edge[](v7)(v23)
\Edge[style = {ultra thick, bend right}](v8)(v17)
\Edge[](v8)(v18)
\Edge[](v8)(v21)
\Edge[](v9)(v16)
\Edge[](v9)(v18)
\Edge[style = {ultra thick}](v9)(v22)
\Edge[](v10)(v17)
\Edge[](v10)(v19)
\Edge[style = {ultra thick, bend right}](v10)(v23)
\Edge[](v11)(v15)
\Edge[](v11)(v20)
\Edge[style = {ultra thick}](v11)(v21)
\Edge[](v12)(v16)
\Edge[style = {ultra thick}](v12)(v19)
\Edge[](v12)(v21)
\Edge[](v13)(v17)
\Edge[style = {ultra thick}](v13)(v20)
\Edge[](v13)(v22)
\Edge[](v14)(v15)
\Edge[style = {ultra thick, bend right}](v14)(v18)
\Edge[](v14)(v23)
\Edge[style = {ultra thick}](v15)(v16)
\Edge[style = {ultra thick}](v15)(v17)
\Edge[style = {ultra thick}](v16)(v17)
\Edge[style = {ultra thick}](v18)(v19)
\Edge[style = {ultra thick}](v18)(v20)
\Edge[style = {ultra thick}](v19)(v20)
\Edge[style = {ultra thick}](v21)(v22)
\Edge[style = {ultra thick}](v21)(v23)
\Edge[style = {ultra thick}](v22)(v23)
\end{tikzpicture}
\caption{On the left is a 4-regular graph $G_1$ such that $G_1^2=K_{15}$.\\  On
the right is a 5-regular graph $G_2$ such that $G_2^2=K_{24}$.~~~~~~~~~~~
}
\end{center}
\end{figure}
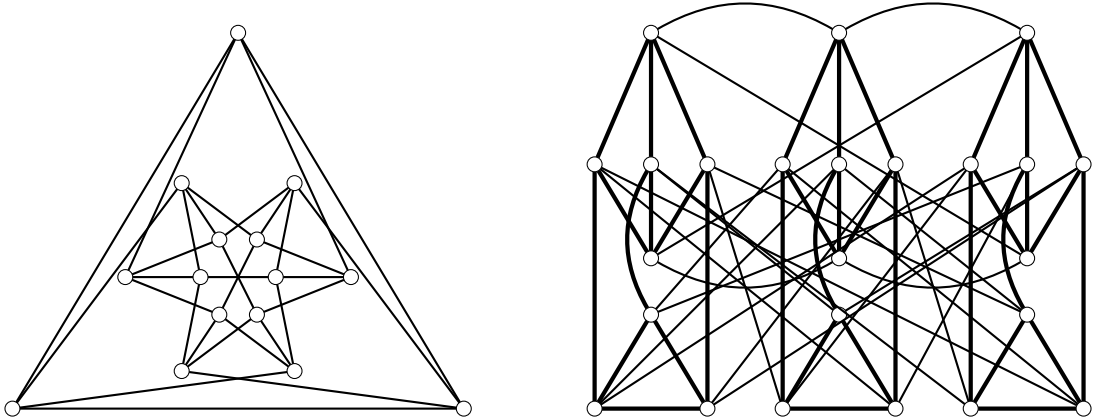

Rather than coloring, or even list coloring, this paper is about \emph{online
list coloring}, a generalization introduced in 2009 by Schauz \cite{Schauz09}
and Zhu \cite{Zhu09online}, and the \emph{online list chromatic number},
$\chi_p$, also called the \emph{paint number}.  We give the definition in
Section \ref{prelims}, but for now if you are unfamiliar with $\chi_p$,
you can substitute $\chil$ (or even $\chi$) and the Main Theorem remains true.
 Our main result is the following.

\begin{mainthm}
If $G$ is a connected graph with maximum degree $\Delta\ge 3$ and
$G$ is not the Peterson graph, the Hoffman-Singleton graph, or a Moore graph
with $\Delta=57$, then $\chi_p(G^2)\le \Delta^2-1$. 
\label{mainthm}
\end{mainthm}

We conclude this section with the following conjecture, which generalizes our
Main Theorem as well as Conjecture \ref{BKlistconj}.

\begin{conj}[Borodin-Kostochka Conjecture (online list coloring version)]
If $G$ is a graph with $\Delta\ge 9$ and $\omega\le \Delta-1$, then $\chiol\le
\Delta-1$.
\label{BKpaint}
\end{conj}

The structure of the paper is as follows.  
In Section \ref{prelims} we give background and definitions.
In Section~\ref{mainproof}, we prove the Main Theorem, subject to a number of
lemmas about forbidden subgraphs in a minimal counterexample.
In Section \ref{lemmas} we prove the lemmas that we deferred in
Section \ref{mainproof}.  Finally, in Section \ref{AT-section}, we generalize
the online list chromatic number to the Alon-Tarsi number, and extend our Main
Theorem to that setting.

\section{Preliminaries}
\label{prelims}

Here we give definitions and background.  Most of our terminology and notation
is standard.  
We write $A\setminus B$ for $A\cap \overline{B}$.  If $H$ is a subgraph of
$G$, then $G\setminus H$ means $G[V(G)\setminus V(H)]$, that is $G$ with the
vertices of $H$ deleted.
For graphs $G$ and $H$, the \emph{join} $G\vee H$ is formed from the disjoint
union of $G$ and $H$ by adding all edges with one endpoint in each of $V(G)$
and $V(H)$.  For any undefined terms, see West \cite{West-IGT}.

A \emph{list size assignment} $f: V(G)\to \Z^+$ assigns to each vertex in $G$ a
list size.  An \emph{$f$-assignment $L$} assigns to each vertex $v$ a subset of
the positive integers $L(v)$ with $|L(v)|=f(v)$.  An $L$-coloring is a proper
coloring $\phi$ such that $\phi(v)\in L(v)$ for all $v$.  A graph $G$ is
\emph{$f$-list colorable} (or \emph{$f$-choosable}) if $G$ has an $L$-coloring
for every $f$-assignment $L$.  In particular, we are interested in the case
where $f(v)=k$ for all $v$ and some constant $k$.  The \emph{list chromatic
number} of $G$ or \emph{choice number} of $G$, denoted $\chil(G)$, or simply
$\chil$ when $G$ is clear from context, is the minimum $k$ such that $G$
is $k$-choosable.
List coloring was introduced by Vizing \cite{Vizing76} and Erd\H{o}s, Rubin, and
Taylor \cite{ErdosRT79} in the 1970s.  Both groups proved the following
extension of Brooks' Theorem.
%
If $G$ is a graph with maximum degree $\Delta\ge 3$ and clique number
$\omega\le \Delta$, then $\chil\le \Delta$.

The next idea we need came about 30 years later.  In 2009,
Schauz~\cite{Schauz09} and Zhu~\cite{Zhu09online} independently introduced
the notion of online list coloring.  This is a variation of list coloring, in
which the list sizes are determined (each vertex $v$ gets $f(v)$ colors), but
the lists themselves are provided
online by an adversary.  

We consider a game between two players, \emph{Lister} and \emph{Painter}.
In round 1, Lister presents the set of all vertices whose lists contain color
1.  Painter must then use color 1 on some independent subset of these
vertices, and cannot change this set in the future.  In each subsequent round
$k$, Lister chooses some subset of the uncolored vertices to contain color $k$
in their lists, and Painter
chooses some independent subset of these vertices to receive color $k$.
Painter wins if he succeeds in painting all vertices.  Alternatively, 
Lister wins if he includes a vertex $v$ among those presented on each of
$f(v)$ rounds, but Painter never paints $v$.

A graph is \emph{online
$k$-list colorable} (or \emph{$k$-paintable}) if Painter can win
whenever $f(v)=k$ for all $v$.  The minimum $k$ such that a graph $G$ is online
$k$-list colorable is its \emph{online list chromatic number}, or \emph{paint
number}, denoted $\chiol$.  A graph is \emph{$d_1$-paintable} if it is
paintable when $f(v)=d(v)-1$ for each vertex $v$.  
In \cite{CranstonR-equiv}, the authors introduced \emph{$d_1$-choosable}
graphs, which are the list-coloring analogue. 
Interest in $d_1$-paintable graphs owes to the fact that none can be induced
subgraphs of a minimal graph with maximum degree $\Delta$ that is not
$(\Delta-1)$-paintable.  In particular, if $G$ is a minimal counterexample to
our Main Theorem, then $G^2$ contains no induced $d_1$-paintable subgraph.  

\begin{lemma}
Let $G$ be a graph with maximum degree $\Delta$ and $H$ be an induced subgraph
of $G$ that is $d_1$-paintable.  If $G\setminus H$ is $(\Delta-1)$-paintable,
then $G$ is $(\Delta-1)$-paintable.
\label{subgraphlemma}
\end{lemma}

\begin{proof}
Let $G$ and $H$ satisfy the hypotheses.  We give an algorithm for Painter to win
the online coloring game when $f(v)=\Delta-1$ for all $v$.  Painter will 
simulate playing two games simultaneously: a game on $G\setminus H$ with
$f(v)=\Delta-1$ and a game on $H$ with $f(v)=d_H(v)-1$.
Let $S_k$ denote the set of vertices presented by Lister on round $k$.  Painter
first plays round $k$ of the game on $G\setminus H$, pretending that Lister
listed the vertices $S_k\setminus H$.  Let $I_k$ denote the independent set of
these that Painter chooses to color $k$.  

Let $S'_k=(S_k\cap V(H))\setminus I_k$, the vertices of $H$ that are in $S_k$
and have no neighbor in $I_k$.  Now Painter plays round $k$ of the game on $H$,
pretending that Lister listed $S'_k$.  Each vertex in $V(G\setminus H)$ will
clearly be listed $\Delta-1$ times.  Consider a vertex $v$ in $V(H)$.  It will
appear in $S_k\setminus S'_k$ for at most $d_G(v)-d_H(v)$ rounds.  So $v$ will
appear in $S'_k$ for at least $(\Delta-1)-(d_G(v)-d_H(v))\ge d_H(v)-1$ rounds.  Now Painter will win both simulated games, and thus win the actual game
on $G$.
\end{proof}

When the graph $G$ in Lemma~\ref{subgraphlemma} is a square, we immediately get
that $G\setminus H$ is $(\Delta-1)$-paintable, as we note in the next lemma.

\begin{lemma}
Let $G$ be a graph with maximum degree $\Delta$ and let $H$ be an induced
subgraph of $G^2$.  If $H$ is $d_1$-paintable, then $G^2$ is $d_1$-paintable. 
If there exists $v$ with $d_{G^2}(v)<\Delta^2-1$, then $G^2$ is
$(\Delta^2-1)$-paintable.
\label{subgraphlemmaG^2}
\end{lemma}

\begin{proof}
We prove the first statement first.  Let $V=V(G)$ and $V_1=V(H)$.
Clearly a graph is $d_1$-paintable only if each component is.  So we 
assume that $G^2[V_1]$ is connected.  For simplicity, we assume also that
$G[V_1]$ is connected.  If not, then some vertex $v$ has neighbors in two or
more components of $G[V_1]$.  We simply add $v$ to $V_1$, since we can color
$v$ first (when it still has at least two uncolored neighbors).

Form $G'$ from $G$ by contracting $G[V_1]$ to a single vertex $r$.  Let $T$ be
a spanning tree in $G'$ rooted at $r$.  Let $\sigma$ be an ordering of the
vertices of $G\setminus H$ by nonincreasing distance in $T$ from $r$. 
Each time that Lister presents a list of vertices, Painter chooses a maximal
independent subset of them, by greedily adding vertices in order $\sigma$.
Each vertex $v\in V\setminus V_1$ is followed in $\sigma$ by the first two
vertices on a path in $T$ from $v$ to $r$.
Thus $v$ will be colored.  We now combine strategies for $G^2\setminus
H$ and $H$ as in the proof of Lemma~\ref{subgraphlemma}.

Now we prove the second statement, which has a similar proof.  
Suppose there exists $v$ with $d_{G^2}(v)< \Delta^2-1$.  As before we order the
vertices by nonincreasing distance in some spanning tree $T$ from $v$, and we
put $v$ and some neighbor $u$ last in $\sigma$.  The difference now is that
even for $u$ and $v$ we are given $\Delta^2-1$ colors.
Since $d_{G^2}(v)<\Delta^2-1$, either (i) $v$ lies on a 3-cycle
or 4-cycle or else (ii) $d_G(v)<\Delta$ or $v$ has some neighbor $u$
with $d_G(u)<\Delta$; in Case (ii), by symmetry we assume $d_G(v)<\Delta$.
In Case (i), $d_{G^2}(u)\le\Delta^2-1$ for some neighbor $u$ of $v$ on the short
cycle and by assumption $d_{G^2}(v)<\Delta^2-1$; so the two final vertices of
$\sigma$ are $u$ and $v$.  In Case (ii), we again have
$d_{G^2}(v)<\Delta^2-1$ and $d_{G^2}(u)\le \Delta^2-1$, so again $u$ and $v$
are last in $\sigma$.  
\end{proof}

The previous lemma implies that $\Delta^2-1\le d_{G^2}(v)\le \Delta^2$ for every
vertex $v$ in a graph $G$ such that $G^2$ is not $(\Delta^2-1)$-paintable.  A
vertex $v$ is \emph{high} if $d_{G^2}(v)=\Delta^2$, and otherwise it is
\emph{low}.  The proof of Lemma~\ref{subgraphlemmaG^2} proves something
slightly more general, which we record in the following corollary.

\begin{cor}
Let $G$ be a graph with maximum degree $\Delta$ and let $H$ be an induced
subgraph of $G^2$.  Let $f(v)=d(v)-1$ for each high vertex of $G^2$ and
$f(v)=d(v)$ for each low vertex.  If $H$ is $f$-paintable, 
then $G^2$ is $(\Delta^2-1)$-paintable. 
\end{cor}

Now we will introduce the Alon-Tarsi Theorem, but we need a few definitions
first.  Let $G$ be a graph and let $\vec{D}$ be a digraph arising by orienting
the edges of $G$. 
A \emph{circulation} is a subgraph of $\vec{D}$ in which each vertex has
equal indegree and outdegree; 
circulations are also called eulerian subgraphs.
The parity of a circulation is the parity of its number of edges.
For a digraph $\vec{D}$, let $EE(\vec{D})$ (resp.  $EO(\vec{D})$) denote the 
set of circulations that are even (resp. odd).

\begin{thmA}[Alon and Tarsi \cite{AlonT92}]
For a digraph $\vec{D}$, if $|EE(\vec{D})|\ne |EO(\vec{D})|$, then $\vec{D}$ is
$f$-choosable, where $f(v)=1+d_{\vec{D}}(v)$ for all $v$.
\end{thmA}

The proof that Alon and Tarsi gave was algebraic and not constructive.  In their
paper, they asked for a combinatorial proof.  This was provided by Schauz
\cite{Schauz-AT}, in the more general setting of paintability.  His proof
relies on an elaborate inductive argument.  The argument does yield a
constructive algorithm, although in general it may run in exponential time.
In \cite{Schauz-nullstellensatz}, Schauz proved an online version of the
combinatorial nullstellensatz from which the paintability version of Alon and
Tarsi's theorem can also be derived.

\begin{thmB}[Schauz \cite{Schauz-AT}]
\label{AT-paint}
For a digraph $\vec{D}$, if $|EE(\vec{D})|\ne |EO(\vec{D})|$, then $\vec{D}$ is
$f$-paintable, where $f(v)=1+d_{\vec{D}}(v)$ for all $v$.
\end{thmB}

Our main result relies heavily on forbidding $d_1$-paintable subgraphs. 
For many of the smaller $d_1$-paintable graphs that we need, we give direct
proofs.  However, for some of the larger $d_1$-paintable graphs, particularly
the classes of unbounded size, our proofs of $d_1$-paintability use Theorem
B.

\section{Proof of Main Theorem}
\label{mainproof}

In this section we prove our main result, subject to a number of lemmas on
forbidden subgraphs, which we defer to the next section.  We typically prove
that a subgraph is forbidden by showing that it is $d_1$-paintable.
If a copy of a subgraph $H$ in $G^2$ contains low vertices, then this
configuration is reducible as long as $H$ is $f$-paintable, where
$f(v)=d_{H}(v)-1$ for each high vertex $v$ and $f(w)=d_{H}(w)$ for each low
vertex $w$.
For many of the graphs, we give an explicit winning strategy for Painter.  In
contrast, for some of the graphs, particularly those of unbounded size,  we
don't give explicit winning strategies.  Instead, we show that they are
$d_1$-paintable via Schauz's extension of the Alon-Tarsi Theorem
(\hyperref[AT-paint]{Theorem B}).

\begin{mainthm}
If $G$ is a connected graph with maximum degree $\Delta\ge 3$ and
$G$ is not the Peterson graph, the Hoffman-Singleton graph, or a Moore graph
with $\Delta=57$, then $\chi_p(G^2)\le \Delta^2-1$. 
\end{mainthm}

\begin{proof}
Let $G$ be a connected graph with maximum degree $\Delta\ge 3$, other than the
graphs excluded in the Main Theorem.  Assume that
$G^2$ is not $(\Delta^2-1)$-paintable.  By Lemma~\ref{subgraphlemmaG^2},
if there exists $v\in V(G)$ with $d_{G^2}(v)<\Delta^2-1$, then $G^2$ is
$(\Delta^2-1)$-paintable.  So $G$ is $\Delta$-regular and has
girth at least 4.  Further, no vertex of $G$ lies on two or more 4-cycles.
It will be helpful in what follows to show that $\omega(G^2)\le \Delta^2-1$.

Clearly $\Delta(G^2)\le \Delta^2$.
Further, $\omega(G^2)=\Delta^2+1$ only if $G^2=K_{\Delta^2+1}$.  
Hoffman and Singleton \cite{HoffmanS60} showed this is possible only if
$\Delta\in \{2,3,7,57\}$; such a graph $G$ is called a Moore graph.  When
$\Delta\in\{2,3,7\}$, the unique realizations
are the 5-cycle, the Peterson graph, and the Hoffman-Singleton graph.  When
$\Delta=57$, no realization is known.  These are precisely the graphs excluded
from the theorem.
Now we consider the case
$\omega(G^2)=\Delta^2$.  Erd\H{o}s, Fajtlowicz, and Hoffman~\cite{ErdosFH80}
showed that the only graph $H$ such that $H^2=K_{\Delta(H)^2}$ is $C_4$. 
Cranston and Kim noted that if $H^2$ is not a clique on at least
$\Delta^2$ vertices, then in fact $\omega(H^2)\le \Delta^2-1$.
For completeness, we reproduce the details.  

Suppose that $\omega(G^2)=\Delta^2$, and let $U$ be the vertices of a maximum
clique in $G^2$.  The result of Erd\H{o}s, Fajtlowicz, and Hoffman implies that
$U$ is not all of $V$.  Choose $v,w\in V$ with $v\in
U$, $w\notin U$ and $v$ adjacent to $w$.  Since $d_{G^2}(v)=\Delta^2$
and $w\notin U$, every neighbor of $w$ must be in $U$.  Applying the
same logic to these neighbors, every vertex within distance 2 of $w$
must be in $U$.  But now we can add $w$ to $U$ to get a larger clique in
$G^2$.  This contradiction implies that in fact $\omega(G^2)\le \Delta^2-1$.

Two vertices are \emph{linked} if they are adjacent in $G^2$, and otherwise they
are \emph{unlinked}.  When we write that vertices are adjacent or nonadjacent,
we mean in $G$; otherwise we write linked or unlinked.
We write $v \adj w$ if $v$ and $w$ are adjacent, and $v\nonadj w$ otherwise.

{\bf Case 1: $G$ has girth 4}

Let $C$ be a 4-cycle with vertices $v_1,\ldots,v_4$, and let $\C=V(C)$.
It is helpful to note that every
$v_i$ is low.  We need two lemmas.  These were first proved in
\cite{CranstonR13+big} for list
coloring, and we generalize them to online list coloring in
Lemmas~\ref{K3vE2} and~\ref{K4vE2}. 
The following two configurations in $G^2$ are reducible:
(A) $K_4\join \overline{K_2}$ where some vertex $w\in V(K_4)$ is low and
(B) $K_3\join \overline{K_2}$ where some vertices $w\in V(K_3)$ and $x\in V(\overline{K_2})$ are both low.

Note that $G^2[\C]\cong K_4$.
This implies that every $w$ adjacent to some $v_i\in \C$ must be linked to all
of $\C$.  Suppose not, and let $w$ be adjacent to $v_1$ and not linked to $v_3$. 
Now $G^2[\C\cup\{w\}] \cong K_3\join \overline{K_2}$, and every $v_i$
is low; this is (B), which is forbidden.  
Now suppose that $w_1$ and $w_2$ are vertices adjacent to $v_i$ and
$v_j$, respectively.  We must have $w_1$ linked to $w_2$, since otherwise
$G^2[\C\cup\{w_1,w_2\}]$ is (A), which is forbidden.  

Now let $x$ be a vertex at distance 2 from $v_1$ and not adjacent to any
$v_i$; let $w_1$ be a common neighbor of $v_1$ and $x$.  Since $w_1$ is linked
to $v_3$, they have a common neighbor $w_3$.  Now $x$ is linked to $v_1$,
$w_1$, and $w_3$.  To avoid configuration (B), $x$ must be linked to all
of $\C$.  Thus, all vertices within distance 2 of $v_1$ must be linked to all
of $\C$.  Now every pair of vertices $x$ and $y$ that are both within
distance 2 of $v_1$ must be linked; otherwise $G^2[\C\cup\{x,y\}]$
is (A).  So the vertices within distance 2 of $v_1$ induce in $G^2$ a clique of
size $\Delta^2$, which contradicts that $\omega(G^2)\le \Delta^2-1$.

{\bf Case 2: $G$ has girth at least 5}

Let $g$ denote the girth of $G$.
First suppose that $g=6$, and let $U$ be the vertices of a 6-cycle.
Note that $G^2[U]=C_6^2$, since girth 6 implies there are no extra edges. 
Since $C_6^2$ is $d_1$-paintable, by Lemma~\ref{C6}, we are done by Lemma~\ref{subgraphlemmaG^2}.

Suppose $g = 7$.
Let $U$ denote the vertices of some $7$-cycle in $G$, with a pendant edge at
a single vertex of the cycle.  Because $G$ has girth 7,
$G^2[U]$ has only the edges guaranteed by its definition.  We show in Lemma~\ref{cycle+pendant} that $G^2[U]$ is
$d_1$-paintable.  So again, we are done by Lemma~\ref{subgraphlemmaG^2}.

Suppose instead that $g\ge 8$.
Let $U = \{v_1, \ldots, v_g, w_1, w_5\}$ be the vertices of some $g$-cycle in
$G$ together with pendant edges $v_1w_1$ and $v_5w_5$.  If $g \ge 9$, then $G^2[U]$ has only the edges guaranteed by its definition.  
If $g = 8$, then $G^2[U]$ has the edges guaranteed by its definition as well as possibly the extra edge $w_1w_5$.
For each girth $g$ at least 8, we show in Lemma~\ref{cycle+2pendant} and Lemma~\ref{cycle+2pendant+edge} that $G^2[U]$ is
$d_1$-paintable.  So again, we are done by Lemma~\ref{subgraphlemmaG^2}.

Now we consider girth 5.  Our approach is similar to that for girth 4, but we
must work harder since we don't necessarily have any low vertices.
Let $C$ be a 5-cycle with vertices $v_1,\ldots,v_5$.  Let $k=\Delta-2$.  For
each $i$, let $V_i$ denote the neighbors of $v_i$ not on $C$.  Let
$\C=V(C)$ and let $\D=\cup_{i=1}^5V_i$.  
Each vertex of $\D$ is linked to either 5, 4, or 3 vertices of $\C$.  We call
these $B_0$-vertices, $B_1$-vertices, and $B_2$-vertices, respectively (a
$B_i$-vertex is unlinked to $i$ vertices of $\C$).  We will consider
four possibilities for the number and location of each type of vertex.  In
each case we find a $d_1$-paintable subgraph.
Let $L$ denote the subgraph $G[\D]$.  Since $G$ has girth 5, we have
$\Delta(L)\le 2$.  Each vertex $w$ with $d_L(w)=2-i$ is a $B_i$-vertex (for
$i\in \{0,1,2\}$).  

Suppose that $G$ has two $B_1$-vertices $w_1$ and $w_2$ and they are unlinked
with distinct vertices in $\C$.  Let $H=G^2[\C\cup\{w_1,w_2\}]$.  If $w_1$ and
$w_2$ are linked, then $H=K_3\join C_4 \supset K_2\join C_4$, which is
$d_1$-paintable, by Lemma~\ref{K2vC4}. If instead $w_1$ and $w_2$ are unlinked,
then $H=K_3\join P_4$, which is also $d_1$-paintable, by Lemma~\ref{K3vP4}.  
So we assume that all $B_1$-vertices are unlinked with the same vertex
$v\in\C$.
As a result, each $B_1$-vertex is an endpoint of a path of length $3 \pmod
5$ in $L$, for otherwise the two endpoints of the path are unlinked with
different vertices in $\C$.  Since the number of odd degree vertices in any
graph is even, here the number of $B_1$-vertices is even. 

{\bf Case 2.1: $G$ has a $B_1$-vertex $w_1$ and a $B_2$-vertex $w_2$.} 

Let $H=G^2[\C\cup\{w_1,w_2\}]$.
Suppose the four vertices of $\C$ linked to $w_1$ include the three vertices
of $\C$ linked to $w_2$.  
If $w_1$ and $w_2$ are linked, then $H=K_3\join P_4$, and if $w_1$ and $w_2$
are unlinked, then $H=K_3\join (K_1+P_3)$.  In each case, $H$ is
$d_1$-paintable, by Lemmas~\ref{K3vP4} and~\ref{K3vK1+P3}, respectively.

Suppose instead that the four vertices of $C$ linked to $w_1$ do not
include all three vertices of $C$ linked to $w_2$.  If $w_1$ is linked with
$w_2$, then $H\supset K_2\join C_4$, which is $d_1$-paintable by
Lemma~\ref{K2vC4}.
If $w_1$ is unlinked with $w_2$, then $H$ is again $d_1$-paintable, by
Lemma~\ref{B1B2}.
Thus, $G^2$ cannot contain both $B_1$-vertices and $B_2$-vertices.  

{\bf Case 2.2: $G$ has no $B_1$-vertices, but only some $B_2$-vertices, and
possibly also $B_0$-vertices.}

Now $L$ consists of disjoint
cycles, each with length a multiple of 5.  This implies that each $V_i$
contains the same number of $B_2$-vertices; by assumption this number is at
least 1.  We call a pair of $B_2$ vertices with distinct cycle neighbors
\emph{near} if their cycle neighbors are adjacent and \emph{far} if their cycle
neighbors are nonadjacent.
If any pair of far $B_2$-vertices are linked, then $G$ has a $d_1$-paintable
subgraph, by Lemma~\ref{farlinked}.  If any pair of near $B_2$-vertices are
linked, then, together with their adjacent cycle vertices, they induce
$K_2\join C_4$, which is $d_1$-paintable by Lemma~\ref{K2vC4}.
Thus, we consider the subgraph induced by $\C$ and 
3 non-successive $B_2$-vertices, say with cycle neighbors
$v_1, v_2, v_4$.  Each such subgraph is $d_1$-paintable, by
Lemma~\ref{3unlinked}.  Combining this with Case 2.1, we conclude that
$G$ contains no $B_2$-vertices.

{\bf Case 2.3: $G$ has $B_1$-vertices and possibly $B_0$-vertices.}

Recall that $G$ has an even number of $B_1$-vertices and they are all unlinked
with the same vertex.  By symmetry, assume that $G$ has $B_1$-vertices $w_2\in
V_2$ and $w_3\in V_3$ and they are both unlinked with $v_5$.  We will find two
disjoint pairs of nonadjacent vertices, such that all four are linked with
$\C-v_5$.

Since $w_3$ is a $B_1$-vertex, it is the endpoint of some path in $L$; let
$w_1\in V_1$ be the neighbor of $w_3$ on this path.  We will show that $w_1$ is
unlinked with some vertex in $\D$.

Recall that $|\D|=5k$.  Suppose that $w_1$ is linked to each vertex of $\D$. 
Since $d_L(w_1)=2$ and $d_L(w_3)=1$, at most 3 of these $5k-1$ vertices linked
with $w_1$ can be reached from $w_1$ by following edges in $L$.  Clearly $w_1$
is linked to the
other $k-1$ vertices of $V_1$.  Now for each vertex $w$ of the remaining
$(5k-1)-3-(k-1)=4k-3$ vertices in $\D$, $w_1$ must have a common neighbor $x$
with $w$ and $x\notin \D\cup \C$.  Furthermore, each such common neighbor $x$
can link $u$ to at most 4 of these vertices (at most one in each other $V_i$,
since the girth is 5).  However, this requires at least $\Ceil{\frac{4k-3}4}=k$
additional neighbors of $w_1$, but we have already accounted for 3 neighbors of
$w_1$.  Thus, $w_1$ is unlinked with some vertex $y\in \D$.  

Let $z$ be a $B_1$ vertex distinct from $y$.  Now $z$ and $v_5$
are unlinked and $w_1$ and $y$ are unlinked.  But every vertex of
$\{w_1,v_5,y,z\}$ is linked to $\C-v_5$.  Thus
$G^2[(\C-v_5)\cup\{w_1,v_5,y,z\}]=K_4\join H$, where $H$ contains disjoint
pairs of nonadjacent vertices.  So $K_4\join H$ is $d_1$-paintable, by 
Lemma~\ref{K4v2E2}.

{\bf Case 2.4: $\D$ has only $B_0$-vertices.} 

Let $H=G^2[\C\cup \D]$.  
We will show that if $H$ is not a clique, then we can choose a different
5-cycle and be in an earlier case.  Suppose that $H$ is not a clique.
Since $\D$ is linked to $\C$ and $G^2[\C]=K_5$, we must have $w_1,w_2\in \D$
with $w_1$ and $w_2$ unlinked.  By symmetry, we have only two cases.

First suppose that $w_1 \in V_1$ and $w_2 \in V_2$ and $w_1$ and $w_2$
are unlinked.  Since $w_1$ is a $B_0$-vertex, we have $w_3 \in V_3$ with
$w_1 \adj w_3$.  Consider the 5-cycle $w_1v_1v_2v_3w_3$.  Now $w_2$ is
not linked to $w_1$, which makes $w_2$ not a $B_0$-vertex for that 5-cycle.
So we are in Case 2.1, 2.2, or 2.3 above.
Now suppose instead that $w_1 \in V_1$ and $w_3 \in V_3$ and $w_1$ and $w_3$
are unlinked.  Now we pick some $w'_3 \in V_3$ with $w_1\adj w'_3$ and consider
the 5-cycle $w_1v_1v_2v_3w'_3$.  Since $w_3$ and $w_1$ are unlinked, $w_3$ is
not a $B_0$-vertex for this 5-cycle, so we are in Case 2.1, 2.2, or 2.3 above.
Hence $G^2[\C \cup \D]$ must be a clique.  

To link all vertices in $\D$, we must have $k(k-1)$ additional vertices in $G$,
at distance 2 from $\C$; call the set of them $\F$.  We see that $|\F|\ge
k(k-1)$ as follows.  
All $5k \choose 2$ pairs of vertices in $\D$ are linked.  The $5{k\choose 2}$
pairs contained within a common $V_i$ are linked via vertices of $\C$.  Each of
the $5k$ vertices is linked with exactly 4 vertices via edges of $L$.  The
remaining links all must be due to vertices of $\F$, and each vertex of $\F$
can link at most ${5 \choose2}=10$ pairs of vertices in $\D$ (at most one
vertex in each $V_i$, since $G$ has girth 5).  Thus $|\F| \ge ({5k\choose 2} -
5{k\choose 2} - 5k(4)/2)/{5 \choose 2} = k(k-1)$.
If any vertex $x\in \F$ has fewer than exactly one neighbor in each $V_i$, then
some pair of vertices in $\D$ will be unlinked.  Thus, each $x\in \F$ has
exactly one neighbor in each $V_i$.  This implies that $\F$ is linked to $\C$,
and hence that $|\F|=k(k-1)$.  We will show that every pair of vertices in $\C
\cup \D \cup \F$ is linked.

Suppose there exists $w\in \D$ and $x\in \F$ with $w$ and $x$ unlinked.
By symmetry, we assume $w\in V_1$.  There exist $w_1\in V_1$ and $w_2\in V_2$
with $x\adj w_1$ and $x\adj w_2$.  Now consider the 5-cycle $xw_1v_1v_2w_2$.
Since $w$ and $x$ are unlinked, $w$ is not a $B_0$-vertex for that 5-cycle.
This puts us in Case 2.1, 2.2., or 2.3 above.  So $\F$ must be linked to $\D$.

Finally suppose there exist $x_1,x_2\in \F$ with $x_1$ and $x_2$ unlinked.
Now there exist $w_1, w_2 \in V_1$ with $x_1\adj w_1$ and $x_2\adj w_2$.
Since $G$ has girth 5, we have $x_1\nonadj w_2$.  And since $x_1$ is linked with
$w_2$, they have some common neighbor $y \in \D\cup \F$.  Now consider the
5-cycle $x_1w_1v_1w_2y$.  Since $x_1$ and $x_2$ are unlinked, $x_2$ is not a
$B_0$-vertex for this 5-cycle.  Hence, we are in Case 2.1, 2.2, or 2.3.

Thus, all vertices of $\C \cup \D \cup \F$ are pairwise linked.  Now $|\C \cup
\D \cup \F| = 5+5k+k(k-1)=k^2+4k+5=(k+2)^2+1=\Delta^2+1$.  This contradicts that
$\omega(G^2)\le \Delta^2-1$ and completes the proof.
\end{proof}

We note that many of the cases of the above proof actually prove that $G^2$ is
$d_1$-paintable, and hence has paint number at most $\Delta(G^2)-1$.  In
particular, this is true when $G$ has girth 6, 7, or at least 9.  Probably with
more work, we could also adapt the proof to the case when $G$ has girth 8.  The
Conjecture that $G^2$ is $(\Delta(G^2)-1)$-paintable unless
$\omega(G^2)\ge\Delta(G^2)$ is a special case of Conjecture~\ref{BKpaint}.
The main obstacle to proving this stronger result is the case when $G$ has girth
at most 5, particularly girth 3 or girth 4.

\section{Proofs of forbidden subgraph lemmas}
\label{lemmas}

In what follows, we slightly abuse the terminology of high and low vertices
defined earlier.  Now a vertex is 
\emph{high} if its list size is one less than its degree and
\emph{low} if its list size equals its degree.  
Note that if a vertex $v$ is high (resp. low) in $G$ by our old definition,
then it will be high (resp. low) in each induced subgraph $H$ by our new
definition.
A vertex is \emph{very low} if
its list size is greater than its degree.  When a vertex $v$ in a graph $G$ is
very low, we may say that we \emph{delete} $v$. If $G-v$ is paintable from
its lists, then so is $G$.  On each round, we play the game on $G-v$ and
consider $v$ after all other vertices, coloring it only if its list contained
the color for that round and we have colored none of its neighbors on that
round.
Recall that $S_k$ denotes the vertices with lists containing color $k$.
We write $E_k$ for the empty graph on
$k$ vertices, i.e., $E_k = \overline{K_k}$.
In what follows, all vertices not specified to be low are assumed to be high.

\subsection{Direct proofs}\label{DirectProofs}
For pictures of the graphs in Lemmas~\ref{K4-e} through~\ref{K3vK1+P3}, see
Figures \ref{indirect-fig1} and \ref{indirect-fig2} in Section
\ref{AT-section}.
\begin{lemma}
\label{K4-e}
If $G$ is $K_4-e$ with one degree 3 vertex high and the other vertices low,
then $G$ is $f$-paintable.
\end{lemma}
\begin{proof}
Let $v_1, v_2$ denote the degree 3 vertices, with $v_1$ low, and let $w_1, w_2$
denote the degree 2 vertices.  If $w_1,w_2\in S_1$, then color them both with 1.
Now the remaining vertices are low and very low, so we can finish.  Otherwise,
color some $v_i$ with 1, choosing $v_2$ if possible.  Now at least one $w_j$
becomes very low and the uncolored $v_k$ is low, so we can finish.
\end{proof}

\begin{lemma}
\label{K3vE2}
If $G$ is $K_3 \join E_2$ with a low vertex in the $K_3$ and a low vertex
in the $E_2$, then $G$ is $f$-paintable.
\end{lemma}
\begin{proof}
Denote the vertices of the $K_3$ by $v_1, v_2, v_3$, with $v_1$ low, and the
vertices of $E_2$ by $w_1, w_2$, with $w_1$ low.
If $w_1,w_2\in S_1$, then color them both 1.  Now $v_1$
becomes very low and $v_2$ and $v_3$ each become low,
so we finish greedily, ending with $v_2$ and $v_1$.
Suppose $w_2 \in S_1$.  If $v_2\in S_1$ 
(or $v_3\in S_1$, by symmetry), then color $v_2$ with 1.  Now $w_1$ becomes
very low (since $S_1 \not\supseteq \{w_1,w_2\}$), and $v_1$ remains
low, so we can finish greedily.  If instead $v_1\in S_1$ and
$v_2, v_3\notin S_1$, then color $v_1$ with 1.  Again 
$w_1$ becomes very low and $v_2$ and $v_3$ become low, so we can finish
greedily.  The situation is similar if $S_1$
contains only a single $w_i$.  Thus, $w_2\notin S_1$.  
Since $S_1\ne \{w_1\}$, some $v_i$ is in $S_1$.  Use color 1 on $v_i$, choosing
$v_2$ or $v_3$ if possible.  What remains is $K_4-e$ with one degree 3 vertex
high and all others low (or very low).  So we finish by Lemma~\ref{K4-e}.
\end{proof}

\begin{lemma}
\label{K4vE2}
If $G$ is $K_4 \join E_2$ with a low vertex in the $K_4$, then $G$ is
$f$-paintable.
\end{lemma}
\begin{proof}
Denote the vertices of the $K_4$ by $v_1, \ldots, v_4$, with $v_1$ low and the
vertices of $E_2$ by $w_1, w_2$.
If $w_1,w_2\in S_1$, then color them both 1.  Now
$v_1$ becomes very low and the other $v_i$ become low, so we can finish by
coloring greedily, with $v_1$ last.  So $S_1$ contains at most one 
$w_i$, say $w_2$.  Suppose $S_1$ contains a $v_j$ other than $v_1$.  Color
$v_j$ with 1.  Now $w_1$ becomes low, $v_1$ remains low, and the
other vertices remain high.  So we can finish the coloring by 
Lemma~\ref{K3vE2}.  If the only $v_i$ in $S_1$ is $v_1$, then color it 1.  Now
the other $v_j$ become low, so again we finish by Lemma~\ref{K3vE2}.  Finally,
if the only vertex in $S_1$ is $w_2$, then color it 1.  Now $v_1$ becomes very
low, and the other $v_i$ become low, so again we can finish by coloring
greedily, ending with a low vertex and a very low vertex.
\end{proof}

\begin{lemma}
\label{K4v2E2}
If $G$ is $K_4 \join H$ with $H$ containing two disjoint nonadjacent pairs, then
$G$ is $d_1$-paintable.
\end{lemma}
\begin{proof}
We may assume $|H|=4$.
Denote the vertices of $K_4$ by $v_1, \ldots, v_4$ and the vertices of $H$ by
$w_1,\ldots, w_4$ with $w_1\nonadj w_2$ and $w_3\nonadj w_4$.  If $w_1,w_2\in
S_1$, then color $w_1$ and $w_2$ with 1.  Now every $v_i$ becomes low, so we
can finish by Lemma~\ref{K4vE2}.  Similarly, if $w_3,w_4\in S_1$.  

If some $v_i$ is missing from $S_1$, then use 1 to color either some $v_j$ or
some $w_k$.  In the first case, we finish by Lemma~\ref{K3vE2} and in the
second by Lemma~\ref{K4vE2}.
So color $v_4$ with 1.  Now, by symmetry, $w_2,w_4\notin S_1$, so they each
become low.
If $w_1, w_2\in S_2$, then color them both with 2.  Now every $v_i$ becomes
low, so we can finish by Lemma~\ref{K3vE2}.  Similarly if $w_3, w_4\in S_2$. 
So $S_2$ contains at most one of $w_1,w_2$ and at most one of $w_3, w_4$.
If $S_2$ contains no $v_i$,  then we color some $w_j$ with 2.  This makes every
$v_i$ low.  Now we can finish by Lemma~\ref{K3vE2}.  So $S_2$ contains some $v_i$, say $v_3$.  

Color $v_3$ with 1.  Recall that
$S_1$ was missing at least one of $w_1, w_2$ and at least one of $w_3, w_4$.
(i) If $w_2,w_4\notin S_2$, then they both become very low, so we can delete
them.  This in turn makes $v_1$ and $v_2$ both very low, so we can finish
greedily. (ii) If $w_2,w_3\notin S_2$, then $w_2$ becomes very low, so we delete
it.  Now $v_1$ and $v_2$ become low; also $w_3$ and $w_4$ are low.  Since
$v_1,v_2,w_3,w_4$ induce $K_4-e$ with all vertices low, we can finish by
Lemma~\ref{K4-e}.  By symmetry, this handles the case $w_1,w_4\notin S_2$.
(iii) If $w_1,w_3\notin S_2$, then the uncolored vertices induce $K_2\join H$,
with all vertices of $H$ low.  Now consider $S_3$.  If $S_3$ contains a
nonadjacent pair in $H$, then color them both 3.  This makes $v_1$ and $v_2$
low, so what remains is $K_4-e$ with all vertices low.  We now finish by
Lemma~\ref{K4-e}.  Similarly, if $S_3$ contains no $v_i$, then color some $w_j$ with
3, and we can finish by Lemma~\ref{K4-e}.  So $S_3$ contains some $v_i$, say
$v_2$, and we color $v_2$ with 3.  Now one of $w_1,w_2$ becomes very low and
one of $w_3,w_4$ becomes very low.  We can delete the very low vertices, which
in turn makes $v_1$ very low.  We can now finish greedily, since what remains
is a 3-vertex path with two low vertices and a very low vertex.
\end{proof}

We won't use Lemma~\ref{K6vE3} in the proof, but it is generally useful so we record it here.

\begin{lemma}
\label{K6vE3}
If $G$ is $K_6 \join E_3$, then $G$ is $d_1$-paintable.
\end{lemma}
\begin{proof}
Denote the vertices of $K_6$ by $v_1,\ldots, v_6$ and the vertices of $E_3$ by
$w_1,w_2,w_3$.  If $w_1,w_2,w_3\in S_1$, then color $w_1,w_2,w_3$ all with 1.
Now all $v_i$ are very low, so we finish greedily.  If no $v_i$ appears in
$S_1$, then color some $w_j$ with 1.  Now all the $v_i$ are low, so we can
finish by Lemma~\ref{K4vE2}.  So some $v_i$ is in $S_1$, say $v_6$.  Color
$v_6$ with 1.  This makes some $w_i$ low, say $w_3$.  Repeating this argument,
we get by symmetry that $v_5\in S_2$ and $S_2$ is missing some $w_j$.  If
$S_2$ is missing $w_3$, then color $v_5$ with 2.  Now $w_3$ becomes very low,
so we delete it.  This in turn makes all uncolored $v_k$ low.  Now we can
finish by Lemma~\ref{K4vE2}.  So instead $S_2$ is missing (by symmetry) $w_2$. 
Again repeating the argument, we must have $v_4\in S_3$ and $w_1\notin S_3$;
otherwise we finish by Lemma~\ref{K3vE2} or Lemma~\ref{K4vE2}.  Now we color
$v_4$ with 3.  What remains is $K_3\join E_3$ with every $w_i$ low.

Now consider $S_4$.  If $w_1,w_2,w_3\in S_4$, then color them all with 3.  Now
all remaining vertices become very low, so we finish greedily.  Suppose instead
that $w_1\in S_4$ and $v_1\notin S_4$.  Color $w_1$ with 4.  What remains is
$K_3\join E_2$ with both $w_i$ low and some $v_j$ low.  So we can finish by
Lemma~\ref{K4-e}.  A similar approach works for any $w_i\in S_4$ and $v_j\notin
S_4$.  So instead, assume by symmetry that $v_1\in S_4$ and $w_1\notin S_4$.
Color $v_1$ with 4.  Now $w_1$ becomes very low, so we delete it.  This in turn
makes $v_2$ and $v_3$ low.  Now we can finish by Lemma~\ref{K4-e}.
\end{proof}

\begin{lemma}
\label{C6}
If $G$ is $C_6^2$, then $G$ is $d_1$-paintable.
\end{lemma}
\begin{proof}
Denote the vertices of the 6-cycle by $v_1,\ldots,v_6$ in order.  So $v_i$ is
adjacent to all but $v_{(i+3)\bmod 6}$.  Consider $S_1$.  If $S_1$ contains some
nonadjacent pair, then color them with 1.  What remains is $C_4$ with all
vertices low, so we can complete the coloring since $C_4$ is 2-paintable.  So
assume that $S_1$ contains no nonadjacent pairs.  Now without loss of
generality, we assume $S_1=\{v_1,v_2,v_3\}$, since adding vertices to $S_1$
only makes things harder to color, as long as $S_1$ induces a clique; we may
also need to permute a nonadjacent pair.
Color $v_1$ with 1.  

Now
$v_5$ and $v_6$ become low.  Consider $S_2$.  Again, if $S_2$ contains a
nonadjacent pair, then we color both vertices with 2 and can finish greedily
since all remaining vertices are low, except for one that is very low.
If $v_2,v_3\in S_2$, then color $v_2$ with 2.  Now $v_6$ becomes very low and
$v_5$ remains low, so we can finish greedily.  So $S_2$ misses at least one of
$v_2,v_3$.  Suppose $v_4\in S_2$.  Color $v_4$ with 3.  What remains is $C_4$.
If $v_2,v_3\notin S_2$, then all vertices are low, and we can finish since $C_4$
is 2-paintable.  Otherwise, $v_5$ or $v_6$ becomes very low and the other
remains low.  Now we can finish greedily.  So $v_4\notin S_2$.  If $v_2\in S_2$,
then color $v_2$ with 2.  Now $v_3$ and $v_4$ become low, so we can finish by
Lemma~\ref{K4-e}.  An analogous argument works if $v_3\in S_2$.  So assume
$v_2,v_3,v_4\notin S_2$.  Now color $v_5$ or $v_6$ with 2.  Again we can finish
by Lemma~\ref{K4-e}.
\end{proof}

\begin{lemma}
\label{K2vC4}
If $G$ is $K_2\join C_4$, then $G$ is $d_1$-paintable.
\end{lemma}
\begin{proof}
Denote the vertices of $K_2$ by $v_1, v_2$ and the vertices of $C_4$ by
$w_1,\ldots w_4$ in order.  If $S_1$ contains a pair of nonadjacent vertices,
then color them both 1.  What remains is $K_4-e$, with all vertices low.  So we
can finish by Lemma~\ref{K4-e}.  So $S_1$ misses at least one of $w_1,w_3$ and
at least one of $w_2,w_4$.  By symmetry, say it misses $w_1$ and $w_2$.  
Suppose $v_1,v_2\notin S_1$. Now by symmetry $w_3\in S_1$, so color $w_3$ with
1.  This makes each of $w_2, v_1, v_2$ low.  
So what remains is $K_3\join E_2$ with two low vertices in the $K_3$ and a low
vertex in the $E_2$.  Hence, we can finish by Lemma~\ref{K3vE2}.

So instead (by symmetry) $v_2\in S_1$.  Color $v_2$ with 1.  What remains is $K_1\join C_4$ with $w_1$
and $w_2$ low.  Consider $S_2$.  Again if $S_2$ contains a nonadjacent pair,
then we color them both 2, and we can finish greedily.  Suppose that $w_3\in
S_2$.  If $w_4\notin S_2$, then we color $w_3$ with 4; now $w_4$ becomes low,
so we can finish by Lemma~\ref{K4-e}.  If instead $w_4\in S_2$, then $w_2\notin
S_2$.  Now when we color $w_3$ with 2, $w_2$ becomes very low, so we can finish
greedily.  So assume $w_3,w_4\notin S_2$.  If $v_1\in S_2$, then color $v_1$
with 1.  What remains is $C_4$ with all vertices low.  Now we can finish the
coloring since $C_4$ is 2-paintable.  The proof is similar to that for
2-choosability, so we omit it.  So assume that $v_1\notin S_2$.  By symmetry, we
have $w_1\in S_2$.  Color $w_1$ with 2.  What remains is $K_4-e$ with only $w_3$
high.  Hence we can finish by Lemma~\ref{K4-e}.
\end{proof}

\begin{lemma}
\label{K3vP4}
If $G$ $K_3\join P_4$, then $G$ is $d_1$-paintable.
\end{lemma}
\begin{proof}
Let $v_1,v_2,v_3$ denote the vertices of $K_3$ and $w_1,\ldots,w_4$ denote the
vertices of the $P_4$ in order.  If $w_1,w_3\in S_1$, then color them both 1.
Now what remains is $K_3\join E_2$ with all but one vertex low, so we can finish
by Lemma~\ref{K3vE2}.  An analagous strategy works if $w_2,w_4\in S_1$.  So
assume $S_1$ misses at least one of $w_1,w_3$ and at least one of $w_2,w_4$.
If $S_1$ misses $v_1$, then use color 1 on some $w_j$, choosing $w_2$ or $w_3$
if possible.  Again, we can finish by Lemma~\ref{K3vE2}.  So assume $v_1\in
S_1$.  Now color $v_3$ with 1.  What remains is $K_2\join P_4$ with at least two
vertices of the $P_4$ low.  Consider $S_2$.  If $w_1,w_3\in S_2$ (or
($w_2,w_4\in S_2$), then color them both 2, and we can finish greedily since all
vertices are low except for one that is very low.  If $v_2\in S_2$, then color
it with 2.  Now in each case we can finish by repeatedly deleting very low
vertices, possibly using Lemma~\ref{K4-e}.  So $v_2\notin S_2$ (and by symmetry
$v_3\notin S_2$).  If possible use color 2 on $w_1$ or $w_4$.  This leaves
$K_3\join E_2$ with enough low vertices to finish by Lemma~\ref{K3vE2}.
Finally, if $w_1,w_4\notin S_2$, then by symmetry $w_2\in S_2$, so color $s_2$
with 2.  What remains contains a $K_4-e$ with all vertices low, so we can finish
by Lemma~\ref{K4-e}.
\end{proof}

\begin{lemma}
\label{K3vK1+P3}
If $G$ is $K_3\join (K_1+P_3)$, then $G$ is $d_1$-paintable.
\end{lemma}
\begin{proof}
Let $v_1, v_2, v_3$ denote the vertices of $K_3$; let $w_1$, $w_2$, $w_3$
denote the vertices of $P_3$ in order, and let $w_4$ be the $K_1$.
If $w_1,w_3\in S_1$, then color them both 1 and we can finish by
Lemma~\ref{K3vE2}.  If instead $w_2,w_4\in S_1$, then color them both 1, and
again we can finish by Lemma~\ref{K3vE2}.  
If $S_1=\{w_4\}$, then color $w_4$ with 1.  What remains is
$K_3\join P_3$ with all vertices of the $K_3$ low.  Since $K_3\join P_3 \cong
K_4 \join E_2$, we can finish by Lemma~\ref{K4vE2}.
If $w_1\in S_1$ (or $w_2\in S_1$ or $w_3\in S_1$) and $v_3\notin S_1$, then color $w_1$ with 1.  Again we can finish by Lemma~\ref{K3vE2}.  
This implies that $v_3\in S_1$.

Since $v_3\in S_1$, color $v_3$ with 1.  Now at least one of $w_1,w_3$ becomes
low and at least one of $w_2,w_4$ becomes low.  What remains is $K_2\join
(K_1+P_3)$, and by symmetry either (i) $w_1$ and $w_2$ are low or (ii) $w_1$ and
$w_4$ are low.  Consider (i).  If we ignore $w_4$, then what remains is
$K_2\join P_3\cong K_3\join E_2$.  Since $w_1$ and $w_2$ are low, we
can finish by Lemma~\ref{K3vE2}.  Instead consider (ii).
If $w_1,w_3\in S_2$, then color them both with 2.
What remains is $K_4-e$ and all vertices are low, so we finish by 
Lemma~\ref{K4-e}.  Suppose instead that $w_2,w_4\in S_2$.  Color them both
with 2, which makes $v_1$ and $v_2$ low.  If $w_1$ became very low, then we
finish greedily.  Otherwise $w_3$ became low, so we finish by Lemma~\ref{K4-e}.
Now suppose $v_1\in S_2$, and color $v_1$ with 2.
We have four possibilities.  If $w_2$ and $w_3$ become low, then we can finish
by Lemma~\ref{K4-e}.  Similarly, if $w_4$ becomes very low, we delete it; now
$v_2$ becomes low, so we can finish by Lemma~\ref{K4-e}.  In the two remaining
cases, we can finish greedily by repeatedly deleting very low vertices.
\end{proof}

\subsection{Proofs via the Alon-Tarsi Theorem}
Our goal in each of the next lemmas is to prove that a certain graph is
$d_1$-paintable. 
For a digraph $\mv{D}$,
we write $\diff(\mv{D})$ to denote $|EE(\mv{D})|-|EO(\mv{D})|$.
In each case we find an orientation $\mv{D}$ such that each vertex has indegree
at least 2 and $\diff(\mv{D})\ne 0$.
Now the Alon-Tarsi Theorem, specifically the generalization in 
\hyperref[AT-paint]{Theorem B}, proves the graph is $d_1$-paintable.
To compute $\diff(\mv{D})$, we typically want to avoid calculating
$|EE(\mv{D})|$ and $|EO(\mv{D})|$ explicitly.  Rather, we look for a
parity-reversing bijection that pairs elements of $EE(\mv{D})$ with elements of
$EO(\mv{D})$.  In computing $\diff(\mv{D})$, we can ignore all circulations
paired by such a bijection.  We also use the following trick to reduce our
work.  We explain it via an example, but it holds more generally.  

Let $\mv{D}$ contain a 5-clique and two
other vertices $w_1$ and $w_2$ such that for each $v$ either $d^+(v)\le 3$ or
$d^+(v)=4$ and $w_1,w_2\in N^+(v)$.  In
computing $\diff(\mv{D})$, we want to restrict the difference to the set of
circulations in which $d^+(w_1)\ge1$ and $d^+(w_2)\ge 1$; call this
$\diff'(\mv{D})$.  By inclusion-exclusion, we have
$\diff'(\mv{D}) =
\diff(\mv{D})-\diff(\mv{D}-w_1)-\diff(\mv{D}-w_2)+\diff(\mv{D}-w_1-w_2)$.
So it suffices to show that the final three terms on the right side are 0.
If any term were nonzero, then, by the Alon-Tarsi Theorem, we would be able to
color the corresponding subgraph from lists of size at most 4.  However, the
subgraph contains a 5-clique, making this impossible.  Thus, each term is 0, and
we have the desired equality.  (In some cases we use a slight variation of this 
approach, instead concluding that the induced subgraph $H$ with $\diff(H)\ne 0$
is $d_1$-paintable.)
Finally, we combine this technique with the
parity-reversing bijection mentioned above, by restricting the bijection only
to the set of circulations where $d^+(w_1)\ge 1$ and $d^+(w_2)\ge 1$.
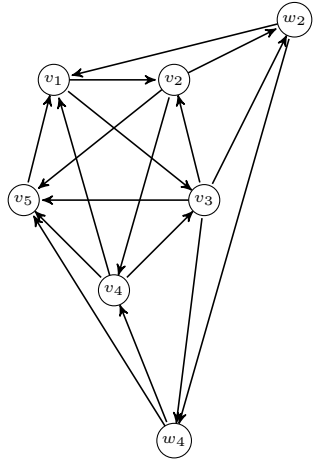
\begin{figure}[ht]
\begin{center}
\begin{tikzpicture}[scale = 8]
\tikzstyle{VertexStyle}=[shape = circle,	
								 minimum size = 6pt,
								 inner sep = 1.2pt,
                                 draw]
\Vertex[x = 0.699999988079071, y = 0.799999997019768, L = \tiny {$v_1$}]{v0}
\Vertex[x = 0.899999976158142, y = 0.799999997019768, L = \tiny {$v_2$}]{v1}
\Vertex[x = 0.949999988079071, y = 0.599999994039536, L = \tiny {$v_3$}]{v2}
\Vertex[x = 0.649999976158142, y = 0.599999994039536, L = \tiny {$v_5$}]{v3}
\Vertex[x = 0.800000011920929, y = 0.449999988079071, L = \tiny {$v_4$}]{v4}
\Vertex[x = 1.10000002384186, y = 0.899999998509884, L = \tiny {$w_2$}]{v5}
\Vertex[x = 0.899999976158142, y = 0.199999988079071, L = \tiny {$w_4$}]{v6}
\Edge[style = {pre}](v3)(v1)
\Edge[style = {post}](v3)(v0)
\Edge[style = {pre}](v1)(v0)
\Edge[style = {post}](v0)(v2)
\Edge[style = {pre}](v0)(v5)
\Edge[style = {post}](v2)(v5)
\Edge[style = {post}](v2)(v1)
\Edge[style = {pre}](v5)(v1)
\Edge[style = {post}](v1)(v4)
\Edge[style = {post}](v4)(v2)
\Edge[style = {pre}](v0)(v4)
\Edge[style = {post}](v4)(v3)
\Edge[style = {pre}](v3)(v2)
\Edge[style = {pre}](v3)(v6)
\Edge[style = {post}](v2)(v6)
\Edge[style = {post}](v6)(v4)
\Edge[style = {post}](v5)(v6)
\end{tikzpicture}
\end{center}
\caption{The orientation for Lemma \ref{farlinked}.}
\label{fig:farlinked}
\end{figure}
\begin{lemma}
\label{farlinked}
Let $H$ be a 5-cycle $v_1,\ldots,v_5$ with pendant edges at $v_2$ and $v_4$, 
leading to vertices $w_2$ and $w_4$, respectively, and let $w_2$ and $w_4$ 
have a common neighbor $x$ (off the cycle).
Let $G=H^2-x$; now $G$ is $d_1$-paintable.
\end{lemma}
\begin{proof}
We orient $G$ to form $\mv{D}$ with 
the following out-neighborhoods:
$N^+(v_1)=\{v_2, v_3\}$,
$N^+(v_2)=\{w_2,v_4,v_5\}$,
$N^+(w_2)=\{v_1,w_4\}$,
$N^+(v_3)=\{v_2,w_2,w_4,v_5\}$,
$N^+(v_4)=\{v_1,v_3,v_5\}$, 
$N^+(w_4)=\{v_4,v_5\}$,
$N^+(v_5)=\{v_1\}$. See Figure \ref{fig:farlinked}.

We will show that $\diff(\mv{D})\ne 0$.  Since each vertex has at least two
in-edges, this proves that $G$ is $d_1$-paintable.  
Let $R=\{{v_3w_2},{v_3w_4}\}$.  For any nonempty subset $S$ of $R$, we
must have $\diff(\mv{D}\setminus R)=0$.  This is because each vertex on the
5-cycle has outdegree at most 3, so will get a list of size at most 4.  And
clearly, we cannot always color $K_5$ from lists of size at most 4.
Thus, it suffices to count the difference, when restricted to the set $A$ of
circulations $\mv{T}$ such that ${v_3w_2},{v_3w_4}\in \mv{T}$.

Let $\mv{T}$ be such a circulation.  
Note that $v_3v_2,v_3v_5\notin \mv{T}$, and thus $v_1v_3,v_4v_3\in \mv{T}$.
Now we consider the 8 possible subsets of $\{w_4v_4, w_4v_5, v_4v_5\}$ in
$\mv{T}$.
Clearly $d^+(w_4)\ge 1$ and $d^-(v_5)\le 1$.  Also,
we can pair the case $w_4v_4,v_4v_5\in \mv{T}$ and $w_4v_5\notin \mv{T}$ with the
case coming from its complement.  Thus, we can restrict to the case when
$w_4v_4\in \mv{T}$ and $v_4v_5\notin \mv{T}$ (and we're not specifying whether
$w_4v_5$ is in or out).  Now consider the directed triangle $v_1v_2,
v_2v_4,v_4v_1$.  We can pair the cases when all or none of these edges are in
$\mv{T}$.  Thus we may assume that either exactly 1 or exactly 2 of these edges
are in.  Considering indegree and outdegree of $v_2$ shows that we must
have $v_1v_2\in \mv{T}$ and $v_2v_4, v_4v_1\notin \mv{T}$.  This implies
$w_2v_1,v_5v_1\in \mv{T}$.
Now we have two ways to
complete $\mv{T}$.  We can have $v_2w_2,w_2w_4,w_4v_5\in \mv{T}$ and $v_2v_4\notin
\mv{T}$ or vice versa.  Each of these gives $|E(\mv{T})|$ odd; thus, we get
$|\diff(D)|=2$.
\end{proof}
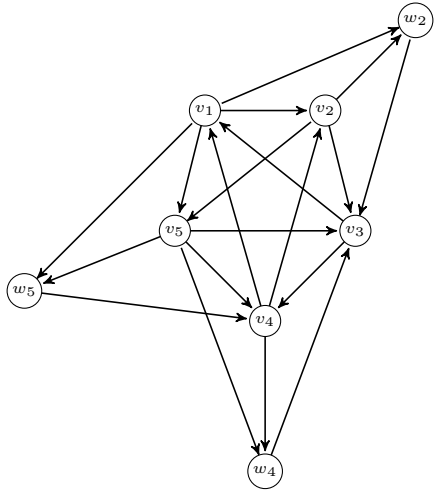
\begin{figure}[ht]
\begin{center}
\begin{tikzpicture}[scale = 8]
\tikzstyle{VertexStyle}=[shape = circle,	
								 minimum size = 6pt,
								 inner sep = 1.2pt,
                                 draw]
\Vertex[x = 0.75, y = 0.699999988079071, L = \tiny {$v_1$}]{v0}
\Vertex[x = 0.949999988079071, y = 0.699999988079071, L = \tiny {$v_2$}]{v1}
\Vertex[x = 1, y = 0.5, L = \tiny {$v_3$}]{v2}
\Vertex[x = 0.699999988079071, y = 0.5, L = \tiny {$v_5$}]{v3}
\Vertex[x = 0.850000023841858, y = 0.350000023841858, L = \tiny {$v_4$}]{v4}
\Vertex[x = 1.10000002384186, y = 0.849999994039536, L = \tiny {$w_2$}]{v5}
\Vertex[x = 0.850000023841858, y = 0.100000023841858, L = \tiny {$w_4$}]{v6}
\Vertex[x = 0.449999988079071, y = 0.399999976158142, L = \tiny {$w_5$}]{v7}
\Edge[style = {post}](v1)(v3)
\Edge[style = {pre}](v1)(v0)
\Edge[style = {pre}](v3)(v0)
\Edge[style = {post}](v0)(v5)
\Edge[style = {pre}](v0)(v2)
\Edge[style = {post}](v5)(v2)
\Edge[style = {pre}](v5)(v1)
\Edge[style = {pre}](v2)(v1)
\Edge[style = {pre}](v1)(v4)
\Edge[style = {pre}](v4)(v2)
\Edge[style = {pre}](v0)(v4)
\Edge[style = {post}](v0)(v7)
\Edge[style = {pre}](v4)(v7)
\Edge[style = {pre}](v4)(v3)
\Edge[style = {pre}](v7)(v3)
\Edge[style = {pre}](v2)(v3)
\Edge[style = {pre}](v2)(v6)
\Edge[style = {post}](v3)(v6)
\Edge[style = {pre}](v6)(v4)
\end{tikzpicture}
\end{center}
\caption{The orientation for Lemma \ref{3unlinked}.}
\label{fig:3unlinked}
\end{figure}
\begin{lemma}
\label{3unlinked}
Let $H$ be a 5-cycle $v_1,\ldots,v_5$ with pendant edges at $v_2$, $v_4$, and
$v_5$, leading to vertices $w_2$, $w_4$, and $w_5$, respectively. 
Let $G=H^2$; now $G$ is $d_1$-paintable.  
\end{lemma}
\begin{proof}
We orient $G$ to form $\mv{D}$ with 
the following out-neighborhoods:
$N^+(v_1)=\{v_2, w_2, v_5, w_5\}$,  $N^+(v_2)=\{w_2,v_3,v_5\}$,
$N^+(w_2)=\{v_3\}$,
$N^+(v_3)=\{v_1,v_4\}$, $N^+(v_4)=\{v_1,v_2,w_4\}$, $N^+(w_4)=\{v_3\}$,
$N^+(v_5)=\{v_3,v_4,w_4,w_5\}$, $N^+(w_5)=\{v_4\}$.  See Figure \ref{fig:3unlinked}.

We will show that $\diff(\mv{D})\ne 0$.  Since each vertex has at least two
in-edges, this proves that $G$ is $d_1$-paintable.  If $\diff(\mv{D}-w_2)\ne 0$,
then we are done, since $\mv{D}-w_2$ is $d_1$-paintable.  
Thus, we can assume that $\diff(\mv{D}-w_2)=0$.  Similarly, we can assume that
$\diff(\mv{D}\setminus S)=0$ for every $S\subseteq \{w_2,w_4,w_5\}$.  Thus, it suffices
to count the difference, when restricted to the set $A$ of circulations such
that $d^+(w_2)=1$, $d^+(w_4)=1$, and
$d^+(w_5)=1$.  Let $\mv{T}$ be such a circulation.  So ${w_2v_3},
{w_4v_3}, {w_5v_4}\in \mv{T}$.  Now $d^+(v_3)=2$, so
${v_3v_1},{v_3v_4}\in \mv{T}$ and $v_2v_3,v_5v_3\notin \mv{T}$.  In
particular, $d^-(v_1)\ge 1$, so $d^+(v_1)\ge 1$.

Now we will pair some circulations in $A$ via a parity-reversing bijection.
Consider the paths ${v_1w_2}$ and ${v_1v_2},{v_2w_2}$.  If a
circulation contains all edges in one path and none in the other, then we can
pair it via a bijection.  The same is true for the paths ${v_1w_5}$ and
${v_1v_5},{v_5w_5}$.  Since $1\le d^+(v_1)\le 2$, and also
$d^-(w_2)=d^-(w_5)=1$, the only way that $\mv{T}$ can avoid these cases is if
either (i) ${v_1v_2},{v_1w_2}\in \mv{T}$ or (ii)
${v_1v_5},{v_1w_5}\in \mv{T}$.  Before we consider these cases, note
that in each case ${v_4v_1}\in \mv{T}$. 

Case (i): 
Now we must have $v_1w_5,v_1v_5\notin \mv{T}$.
Note that ${v_2w_2}\notin \mv{T}$, which implies
${v_4v_2}\notin \mv{T}$.  Also ${v_2v_5}\in \mv{T}$.  Further,
$d^-(w_5)=1$ implies ${v_5w_5}\in \mv{T}$, which in turn yields
${v_5v_4}, {v_5w_4}\notin \mv{T}$.  Finally, ${v_4w_4}\in \mv{T}$. 
Thus, we have a unique $\mv{T}$ (with an odd number of edges).

Case (ii): 
Now we must have $v_1w_2,v_1v_2\notin \mv{T}$ and also $v_5w_5\notin
\mv{T}$.
Note that ${v_2w_2}\in \mv{T}$, which implies that
${v_4v_2}\in \mv{T}$ and also that ${v_2v_5}\notin \mv{T}$.  Now we
get that either (a) ${v_5v_4}\in \mv{T}$, and thus ${v_4w_4}\in
\mv{T}$ and ${v_5w_4}\notin \mv{T}$ or else (b) ${v_5w_4}\in \mv{T}$
and ${v_5v_4}, {v_4w_4}\notin \mv{T}$.  Again, by a parity-reversing
bijection, we see that together these circulations contribute nothing to
$\diff(A)$ (in fact there is only one of each).  Now combining Cases (i) and
(ii), we get that $|\diff(A)|=1$, and in fact $|\diff(\mv{D})|=1$.  Thus,
$G$ is $d_1$-paintable.
\end{proof}
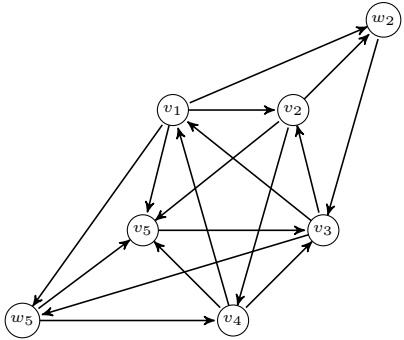
\begin{figure}[ht]
\begin{center}
\begin{tikzpicture}[scale = 8]
\tikzstyle{VertexStyle}=[shape = circle,	
								 minimum size = 6pt,
								 inner sep = 1.2pt,
                                 draw]
\Vertex[x = 0.75, y = 0.799999997019768, L = \tiny {$v_1$}]{v0}
\Vertex[x = 0.949999988079071, y = 0.799999997019768, L = \tiny {$v_2$}]{v1}
\Vertex[x = 1, y = 0.599999994039536, L = \tiny {$v_3$}]{v2}
\Vertex[x = 0.699999988079071, y = 0.599999994039536, L = \tiny {$v_5$}]{v3}
\Vertex[x = 0.850000023841858, y = 0.449999988079071, L = \tiny {$v_4$}]{v4}
\Vertex[x = 1.10000002384186, y = 0.949999999254942, L = \tiny {$w_2$}]{v5}
\Vertex[x = 0.5, y = 0.449999988079071, L = \tiny {$w_5$}]{v6}
\Edge[style = {post}](v1)(v3)
\Edge[style = {pre}](v1)(v0)
\Edge[style = {pre}](v3)(v0)
\Edge[style = {post}](v0)(v5)
\Edge[style = {pre}](v0)(v2)
\Edge[style = {post}](v5)(v2)
\Edge[style = {pre}](v5)(v1)
\Edge[style = {post}](v2)(v1)
\Edge[style = {post}](v1)(v4)
\Edge[style = {post}](v4)(v2)
\Edge[style = {pre}](v0)(v4)
\Edge[style = {post}](v0)(v6)
\Edge[style = {pre}](v4)(v6)
\Edge[style = {post}](v4)(v3)
\Edge[style = {post}](v6)(v3)
\Edge[style = {pre}](v2)(v3)
\Edge[style = {pre}](v6)(v2)
\end{tikzpicture}
\end{center}
\caption{The orientation for Lemma \ref{B1B2}.}
\label{fig:B1B2}
\end{figure}
\begin{lemma}
\label{B1B2}
Let $H$ be a 5-cycle $v_1,\ldots,v_5$ with pendant edges at $v_2$ and $v_5$, 
leading to vertices $w_2$ and $w_5$, respectively, and let $w_5$ and $v_3$ 
have a common neighbor $x$ (off the cycle).
Let $G=H^2-x$; now $G$ is $d_1$-paintable.
\end{lemma}
\begin{proof}
We orient $G$ to form $\mv{D}$ with 
the following out-neighborhoods:
$N^+(v_1)=\{v_2, w_2, v_5, w_5\}$,
$N^+(v_2)=\{w_2,v_4,v_5\}$,
$N^+(w_2)=\{v_3\}$,
$N^+(v_3)=\{v_1, v_2, w_5\}$,
$N^+(v_4)=\{v_1,v_3,v_5\}$, 
$N^+(v_5)=\{v_3\}$, 
$N^+(w_5)=\{v_4,v_5\}$.  See Figure \ref{fig:B1B2}.

We will show that $\diff(\mv{D})\ne 0$.  Since each vertex has at least two
in-edges, this proves that $G$ is $d_1$-paintable.  
Note that for each nonempty subset $S\subseteq \{w_2,w_5\}$, we have
$\diff(\mv{D}\setminus S)=0$, since otherwise we can color the corresponding
subgraph from lists of size 4, even though it contains a 5-clique.
So by inclusion-exclusion, we can restrict our count of $\diff$ to the
set of circulations $A$ where $w_2$ and $w_5$ each have positive indegree.  
Consider the paths ${v_1w_2}$ and ${v_1v_2},{v_2w_2}$.  
Let $\mv{T}$ be a circulation in $A$.
If $T$ contains all edges of one path and none of the other, then we can pair it
via a parity-reversing bijection.  So we assume we are not in these situations.
 Since $w_2$ has positive indegree,
and hence indegree 1, we either have 
(i) ${v_1w_2},{v_1v_2}\in \mv{T}$ and ${v_2w_2}\notin \mv{T}$ 
or 
(ii) ${v_2w_2}\in \mv{T}$ and ${v_1w_2},{v_1v_2}\notin \mv{T}$.

Case (i): ${v_1w_2},{v_1v_2}\in \mv{T}$ and ${v_2w_2}\notin \mv{T}$.
Clearly ${w_2v_3}\in \mv{T}$.  Since $d^+(v_1)=2$, we have
${v_3v_1},{v_4v_1}\in \mv{T}$ and ${v_1v_5},{v_1w_5}\notin
\mv{T}$.  Suppose ${v_3v_2}\in \mv{T}$. Now also ${v_2v_4},{v_2v_5},
{v_5v_3}\in \mv{T}$.  Finally, since $w_5$ has positive indegree,
${v_3w_5},{w_5v_4},{v_4v_3}\in \mv{T}$.  The resulting
circulation is even.  Suppose instead that ${v_3v_2}\notin \mv{T}$.  If
${v_2v_5}\in \mv{T}$, then we get ${v_5v_3}, {v_3w_5}, {w_5v_4}\in \mv{T}$.
The resulting circulation is odd.  If instead ${v_2v_5}\notin \mv{T}$ and
${v_2v_4}\in \mv{T}$, then we have three possibilities to ensure
$d^+(w_5)>0$.  Either ${v_3w_5},{w_5v_4},{v_4v_5},{v_5v_3}\in
\mv{T}$ or ${v_3w_5}, {w_5v_4},{v_4v_3}\in \mv{T}$ or
${v_3w_5},{w_5v_5},{v_5v_3}\in \mv{T}$.  Two of the resulting
circulations are odd and one is even.  Thus in total for Case (i), we have one
more odd circulation than even.

Case (ii): ${v_2w_2}\in \mv{T}$ and ${v_1w_2},{v_1v_2}\notin \mv{T}$.
We have ${v_2w_2}\in \mv{T}$, which implies ${w_2v_3}\in
\mv{T}$ and ${v_3v_2}\in \mv{T}$.  This further yields
${v_2v_4},{v_2v_5}\notin \mv{T}$.
Again we will pair some of the
circulations in $A$ via a parity-reversing bijection.
Consider the paths ${v_3w_5}$ and ${v_3v_1},{v_1w_5}$.  If a
circulation contains all edges in one path and none in the other, then we can
pair it via a bijection.  Since $1\le d^-(w_5)$, the only way that $\mv{T}$
can avoid these cases is if
either (a) ${v_1w_5}\in \mv{T}$ and ${v_3v_1}\notin \mv{T}$
or (b) ${v_3v_1}\in \mv{T}$ and ${v_1w_5}\notin \mv{T}$ (and thus
${v_3w_5}\in \mv{T}$ or (c) ${v_3v_1}, {v_1w_5},{v_3w_5} \in
\mv{T}$.  Consider (a).  ${v_1w_5}\in \mv{T}$ implies ${v_4v_1}\in
\mv{T}$, and thus ${w_5v_4}\in \mv{T}$.  We also have the option of
all or none of ${v_3w_5},{w_5v_5},{v_5v_3}$ in $\mv{T}$.  One of the
resulting circulations is odd and the other is even.  Consider (b).  Now
${v_3v_1}\in \mv{T}$ and ${v_1w_5}\notin \mv{T}$ imply ${v_1v_5}\in
\mv{T}$, and thus ${v_5v_3}\in \mv{T}$.  Now $d^+(w_5)>0$ implies
${v_3w_5}, {w_5v_4},{v_4v_3}\in \mv{T}$.  The resulting circulation is
odd.  Consider (c).  Now we get ${w_5v_5}\in \mv{T}$, which implies
${v_5v_3}\in \mv{T}$.  We also get ${w_5v_4}\in \mv{T}$, which
implies $\mv{v_4v_3}\in \mv{T}$.  The resulting circulation is even.  Thus
in total for Case (ii), we have the same number of even and odd circulations.

So combining Cases (i) and (ii), we have one more odd circulation than even.
Thus $\diff(\mv{D})\ne 0$, so $G$ is $d_1$-paintable.
\end{proof}

Form $\mv{P_n}$ from $(P_n)^2$ by orienting all edges from left to right. 
Number the vertices as $v_1,\ldots, v_n$ from left to right.  
A subgraph $\mv{T}\subseteq \mv{P_n}$ is \emph{weakly eulerian} if each vertex
$w\notin\{v_1,v_n\}$ satisfies $d^+(w)=d^-(w)$ and 
$d^+(v_1)=d^-(v_n)=i$ for some $i\in\{1,2\}$.
Let $EE_i(\mv{P_n})$ (resp. $EO_i(\mv{P_n}$))
denote the set of even (resp. odd) weakly eulerian subgraphs where
$d^+(v_1)=d^-(v_n)=i$.  Finally, let $f_i(n)=|EE_i(\mv{P_n})|-|EO_i(\mv{P_n})|$.
We will not apply the following lemma directly to find $d_1$-paintable
subgraphs.  However, it will be helpful in the proof for the remaining
$d_1$-paintable graph, which includes cycles of arbitrary length.

\begin{lemma}
If $n=3k+j$ for some positive integer $k$ and $j\in\{-1,0,1\}$,
then $f_1(n)=j$ and for $n\ge 4$ also $f_2(n)=-f_1(n-2)$, with $f_i(n)$ as
defined above.
\label{path-lemma}
\end{lemma}

\begin{proof}
Rather than directly counting weakly eulerian subgraphs, we again use a
parity-reversing bijection.
We first prove that $f_2(n)=-f_1(n-2)$.  The complement of
each $\mv{D}\in EE_2(\mv{P_n})\cup EO_2(\mv{P_n})$ has
$d^+(v_2)=d^-(v_{n-1})=1$ and $d^+(w)=d^-(w)$ for each
$w\notin\{v_1,v_2,v_{n-1},v_n\}$ (and
$d^+(v_1)=d^-(v_n)=d^-(v_2)=d^+(v_{n-1})=0$). 
Since $\mv{P_n}$ has $2n-3$ edges, 
each digraph has parity opposite its complement;
so $f_2(n)=-f_1(n-2)$.

Now we determine $f_1(n)$.
Let $\mv{T}$ be a weakly eulerian subgraph with $d^+(v_1)=1$.  
Consider the directed paths $v_1v_3$ and $v_1v_2,v_2v_3$.
If $\mv{T}$ contains all of one path and none of the other, then we can pair
$\mv{T}$ with its complement, which has opposite parity.
If neither of these cases holds, then we must have
${v_1v_2},{v_2v_4}\in \mv{T}$ and ${v_1v_3},{v_2v_3}\notin
\mv{T}$.  
This yields $f_1(n)=f_1(n-3)$.  It remains only to check that
$f_1(2)=-1$, $f_1(3)=0$, and $f_1(4)=1$.
\end{proof}

\begin{figure}[ht]
\begin{center}
\begin{tikzpicture}[scale = 8]
\tikzstyle{VertexStyle}=[shape = circle,	
								 minimum size = 6pt,
								 inner sep = 1.2pt,
                                 draw]
\Vertex[x = 1.25, y = 0.699999988079071, L = \tiny {$v_1$}]{v0}
\Vertex[x = 1.45000004768372, y = 0.550000011920929, L = \tiny {$v_2$}]{v1}
\Vertex[x = 1.04999995231628, y = 0.550000011920929, L = \tiny {$v_7$}]{v2}
\Vertex[x = 1.5, y = 0.350000023841858, L = \tiny {$v_3$}]{v3}
\Vertex[x = 1, y = 0.350000023841858, L = \tiny {$v_6$}]{v4}
\Vertex[x = 1.39999997615814, y = 0.149999976158142, L = \tiny {$v_4$}]{v5}
\Vertex[x = 1.25, y = 0.849999994039536, L = \tiny {$u$}]{v6}
\Vertex[x = 1.10000002384186, y = 0.149999976158142, L = \tiny {$v_5$}]{v7}
\Edge[style = {post}](v0)(v1)
\Edge[style = {post}](v0)(v3)
\Edge[style = {post}](v0)(v6)
\Edge[style = {post}](v1)(v3)
\Edge[style = {post}](v1)(v5)
\Edge[style = {post}](v1)(v6)
\Edge[style = {post}](v2)(v0)
\Edge[style = {post}](v2)(v1)
\Edge[style = {post}](v3)(v5)
\Edge[style = {post}](v3)(v7)
\Edge[style = {post}](v4)(v2)
\Edge[style = {post}](v5)(v4)
\Edge[style = {post}](v5)(v7)
\Edge[style = {post}](v6)(v2)
\Edge[style = {post}](v7)(v2)
\Edge[style = {post}](v7)(v4)
\Edge[style = {post}](v4)(v0)
\end{tikzpicture}
\end{center}
\caption{The orientation for Lemma \ref{cycle+pendant} with $n=7$.}
\label{fig:cycle+pendant}
\end{figure}
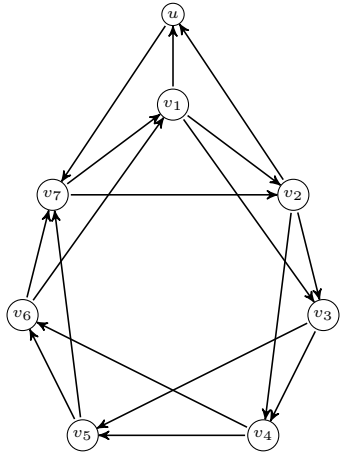

\begin{lemma}
\label{cycle+pendant}
Cycle + one pendant edge:
Let $J_n$ consist of an $n$-cycle on vertices $v_1,\ldots,v_n$ (in clockwise
order) with a pendant edge at $v_1$ leading to vertex $u$.  Form $\mv{D_n}$
by squaring $J_n$ and orienting the edges as follows. 
Orient edges ${v_iv_{i+1}}$ and ${v_iv_{i+2}}$ away from $v_i$ (with
subscripts modulo $n$).  Orient ${uv_n}$ away from $u$ and ${v_1u}$
and ${v_2u}$ toward $u$.  We will show that $\diff(\mv{D_n})\ne 0$
when $n\not\equiv 2 \bmod 3$ (or else $f(\mv{D_n}-u)\not\equiv 0$). 
\end{lemma}
\begin{proof}
Form $\mv{D_n}$ as in the lemma.
We will show that $\diff(\mv{D_n})\ne 0$, and thus $J^2_n$ is $d_1$-paintable.
We may assume that $\diff(\mv{D_n-u})\ne 0$, 
for otherwise $\mv{D_n-u}$ is $d_1$-paintable.  Thus, restricting
our count to the set $A$ of circulations with $d^+(u)=1$ does not affect the
difference.
Let $\mv{T}$ be a circulation in $A$.
Consider the directed paths $v_1u$ and $v_1v_2, v_2u$.
If $\mv{T}$ contains all edges of one path and none of the other, then we can
pair $\mv{T}$ via a parity-reversing bijection.  So we assume we are not in one
of those cases.  Clearly $\mv{T}$ contains $\mv{uv_n}$ and exactly one of
${v_1u}$ and ${v_2u}$.  Thus either 
(i) ${v_2u}\in \mv{T}$ and ${v_1u},{v_1v_2}\notin\mv{T}$ 
or 
(ii) ${v_1u},{v_1v_2}\in \mv{T}$ and ${v_2u}\notin\mv{T}$.

Case (i): ${v_2u}\in \mv{T}$ and ${v_1u},{v_1v_2}\notin\mv{T}$.
Since ${v_2u}\in \mv{T}$ and ${v_1v_2}\notin \mv{T}$, we must have
${v_nv_2}\in \mv{T}$ and ${v_2v_3}, {v_2v_4} \notin \mv{T}$.  By
removing edges ${uv_n},{v_nv_2},{v_2u}$, we see that these circulations
are in bijection with the circulations in $\mv{D}_n-u-v_2$ (with the parity of
each subgraph reversed).  If we exclude the empty graph, these circulations are
in bijection with those counted by $f_1(n-1)$, since $d^+(v_1)=1$ and
$d^-(v_3)=1$.
Adding 1 for the empty subgraph, this difference is $1-f_1(n-1)$, and when we
account for removing edges $uv_n,v_nv_2,v_2u$, the difference is $-1+f_1(n-1)$.

Case (ii): ${v_1u},{v_1v_2}\in \mv{T}$ and ${v_2u}\notin\mv{T}$.
Since ${v_1u},{v_1v_2}\in\mv{T}$, we must have ${v_{n-1}v_1},
{v_nv_1}\in \mv{T}$ and ${v_1v_3}\notin\mv{T}$.
After removing edges ${v_nv_1},{v_1u},{uv_n}$, we see that these
circulations are in bijection with the circulations in
$\mv{D}_n-u-{v_nv_1}-{v_1v_3}$ that contain edges 
${v_{n-1}v_1}$ and ${v_1v_2}$.  We will count the difference of these even
and odd circulations, then multiply the total by $-1$ (to account for
removing edges ${v_1u},{uv_n},{v_nv_1}$) before adding to the total
above.

We consider two subcases:
${v_nv_2}\notin\mv{T}$ and ${v_nv_2}\in\mv{T}$.  In the first case, these
circulations are in bijection with circulations of
$\mv{D}_{n-1}-u-v_1$ (since $d^+(v_n)=0$ and $v_1$ may be suppressed).
This difference is counted by $f_1(n-2)$.
In the second case, 
the difference is counted by $-f_2(n)$, since we may think of deleting
${v_1v_2}$ and replacing ${v_nv_2}$ with ${v_nv_1}$; our path now 
starts at $v_2$ and runs through $v_n$ to $v_1$ (and the parity is changed when
accounting for ${v_1v_2}$).

Thus, the total difference in Case (ii) is counted by $f_1(n-2)-f_2(n)$.
Thus, the total difference overall is counted by $-1+f_1(n-1)
-f_1(n-2)+f_2(n)=-1+f_1(n-1)-2f_1(n-2).$  Substituting values from
Lemma~\ref{path-lemma} shows that this expression is non-zero when
$n\not\equiv 2\bmod 3$.  
\end{proof}

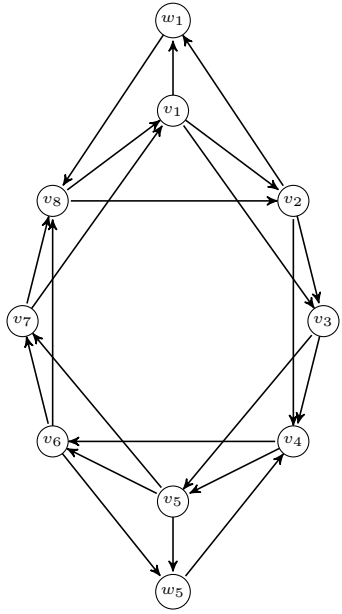
\begin{figure}[ht]
\begin{center}
\begin{tikzpicture}[scale = 8]
\tikzstyle{VertexStyle}=[shape = circle,	
								 minimum size = 6pt,
								 inner sep = 1.2pt,
                                 draw]
\Vertex[x = 1.25, y = 0.699999988079071, L = \tiny {$v_1$}]{v0}
\Vertex[x = 1.04999995231628, y = 0.550000011920929, L = \tiny {$v_8$}]{v1}
\Vertex[x = 1.45000004768372, y = 0.550000011920929, L = \tiny {$v_2$}]{v2}
\Vertex[x = 1, y = 0.350000023841858, L = \tiny {$v_7$}]{v3}
\Vertex[x = 1.5, y = 0.350000023841858, L = \tiny {$v_3$}]{v4}
\Vertex[x = 1.04999995231628, y = 0.149999976158142, L = \tiny {$v_6$}]{v5}
\Vertex[x = 1.45000004768372, y = 0.149999976158142, L = \tiny {$v_4$}]{v6}
\Vertex[x = 1.25, y = 0.849999994039536, L = \tiny {$w_1$}]{v7}
\Vertex[x = 1.25, y = -0.100000023841858, L = \tiny {$w_5$}]{v8}
\Vertex[x = 1.25, y = 0.050000011920929, L = \tiny {$v_5$}]{v9}
\Edge[style = {post}](v1)(v2)
\Edge[style = {pre}](v1)(v7)
\Edge[style = {post}](v1)(v0)
\Edge[style = {post}](v2)(v7)
\Edge[style = {pre}](v2)(v0)
\Edge[style = {pre}](v7)(v0)
\Edge[style = {pre}](v0)(v3)
\Edge[style = {post}](v3)(v1)
\Edge[style = {post}](v0)(v4)
\Edge[style = {pre}](v4)(v2)
\Edge[style = {pre}](v1)(v5)
\Edge[style = {post}](v5)(v3)
\Edge[style = {post}](v2)(v6)
\Edge[style = {pre}](v6)(v4)
\Edge[style = {post}](v9)(v5)
\Edge[style = {pre}](v9)(v6)
\Edge[style = {post}](v4)(v9)
\Edge[style = {post}](v6)(v5)
\Edge[style = {post}](v9)(v3)
\Edge[style = {pre}](v8)(v9)
\Edge[style = {pre}](v8)(v5)
\Edge[style = {post}](v8)(v6)
\end{tikzpicture}
\end{center}
\caption{The orientation for Lemma \ref{cycle+2pendant} with $n=8$.}
\label{fig:cycle+2pendant}
\end{figure}

\begin{lemma}
\label{cycle+2pendant}
Cycle + two pendant edges:
For $n \ge 7$, let $J_n$ consist of an $n$-cycle on vertices $v_1,\ldots,v_n$ (in clockwise
order) with pendant edges at $v_1$ and $v_5$ leading to vertices $w_1$ and
$w_5$.  Form $\mv{D}_n$ by squaring $J_n$ and orienting the edges as follows. 
Orient edges
${v_iv_{i+1}}$ and
${v_iv_{i+2}}$ away from $v_i$ (with subscripts modulo $n$). 
Orient
${w_1v_n}$ away from $w_1$ and ${v_1w_1}$ and ${v_2w_1}$ toward $w_1$; 
similarly, orient
${w_5v_4}$ away from $w_5$ and ${v_5w_5}$ and ${v_6w_5}$ toward $w_5$. 
We will show that $f(\mv{D_n})\ne 0$ (or else $f(\mv{D_n}\setminus B)\ne 0$ for some subset $B\subseteq \{w_1,w_5\}$).
\end{lemma}

\begin{proof}
Form $\mv{D_n}$ as in the lemma.
We will show that $\diff(\mv{D_n})\ne 0$, and thus $J^2_n$ is $d_1$-paintable.
For each nonempty $B\subseteq \{w_1,w_5\}$,
we may assume that $\diff(\mv{D_n}\setminus B)= 0$, 
for otherwise $\mv{D_n}\setminus B$ is $d_1$-paintable.  
Thus, restricting our count to the set $A$ of circulations with $d^+(w_1)=1$
and $d^+(w_5)=1$ does not affect the difference.  

Let $\mv{T}$ be a circulation in $A$.
Clearly $\mv{T}$ contains $\mv{w_1v_n}$ and exactly one of $\mv{v_1w_1}$
and $\mv{v_2w_1}$.
Consider the directed paths $v_1w_1$ and $v_1v_2, v_2w_1$.
If $\mv{T}$ contains all edges of one path and none of the other, then we can
pair $\mv{T}$ via a parity-reversing bijection.  So we assume we are not in one
of those cases.  
Thus either (i) ${v_2w_1}\in \mv{T}$ and
${v_1w_1},{v_1v_2}\notin\mv{T}$ or (ii) ${v_1w_1},{v_1v_2}\in \mv{T}$
and ${v_2w_1}\notin\mv{T}$.

Now we consider the directed paths ${v_5w_5}$ and ${v_5v_6}, {v_6w_5}$.  Among
those circulations, within Cases (i) and (ii), where 
$\mv{T}$ contains all of one path and none of the other
we again pair $\mv{T}$ via a parity-reversing bijection, by 
removing the edges of one path and adding the edges of the other.
Thus, we need only consider two subcases
in each case: (1) ${v_6w_5}\in \mv{T}$ and ${v_5w_5},{v_5v_6}\notin\mv{T}$ and
(2) ${v_5w_5},{v_5v_6}\in \mv{T}$ and ${v_6w_5}\notin\mv{T}$.

Case (i.1): 
${v_2w_1}\in \mv{T}$ and ${v_1w_1},{v_1v_2}\notin\mv{T}$ and also
${v_6w_5}\in \mv{T}$ and ${v_5w_5},{v_5v_6}\notin\mv{T}$.
Since ${v_2w_1} \in \mv{T}$, we must have ${v_nv_2} \in \mv{T}$ and also
${v_2v_3}, {v_2v_4} \not \in \mv{T}$. 
Similarly, since ${v_6w_5} \in \mv{T}$, we must have ${v_4v_6} \in \mv{T}$ and
also ${v_6v_7}, {v_6v_8} \not \in \mv{T}$. 
Since both triangles $w_1v_nv_2$ and $v_4v_6w_5$ must be included in every circulation under consideration, we may remove $w_1, v_2, w_5, v_6$ without changing the total difference.
Now any non-empty circulation must contain both $v_1v_3$ and $v_5v_7$.  But we have a parity reversing bijection between those circulations containing $v_3v_5$ and those containing $v_3v_4, v_4v_5$, so for non-empty circulations the difference is zero.  Thus after adding in the empty circulation, we see that the total difference is $1$ for this case.

Case (i.2): 
${v_2w_1}\in \mv{T}$ and ${v_1w_1},{v_1v_2}\notin\mv{T}$ and also
${v_5w_5},{v_5v_6}\in \mv{T}$ and ${v_6w_5}\notin\mv{T}$.
Since ${v_2w_1} \in \mv{T}$, we must have ${v_nv_2} \in \mv{T}$ and hence
${v_2v_3}, {v_2v_4} \not \in \mv{T}$. Since the triangle $w_1v_nv_2$ must be
included in every circulation under consideration, we may remove $w_1, v_2$ at
the cost of negating the difference.  Since ${v_5w_5},{v_5v_6}\in \mv{T}$, we
must have ${w_5v_4}, {v_3v_5}, {v_4v_5}\in \mv{T}$ and ${v_5v_7} \not \in
\mv{T}$.  But then ${v_3v_4} \not \in \mv{T}$ and hence ${v_4v_6} \not \in
\mv{T}$.  Now we may remove $w_5$ and $v_4$ at the cost of negating the
difference again.  Now removing $v_3$ and $v_5$ we lose three edges that must
be in every circulation and the resulting difference is counted by $f_1(n-4)$;
the paths run from $v_6$ through $v_n$ to $v_1$. 
Hence this case contributes $-f_1(n-4)$ to the difference.

Case (ii.1): 
${v_1w_1},{v_1v_2}\in \mv{T}$ and ${v_2w_1}\notin\mv{T}$ and also
${v_6w_5}\in \mv{T}$ and ${v_5w_5},{v_5v_6}\notin\mv{T}$.
Since $v_1w_1,v_1v_2\in \mv{T}$, we get $v_nv_1, v_{n-1}v_1\in \mv{T}$.
Since $v_6w_5\in \mv{T}$ and $v_5v_6\notin \mv{T}$, we get $v_4v_5\in \mv{T}$
and $v_6v_7,v_6v_8\notin \mv{T}$.  Since we have $v_{n-1}v_1\in \mv{T}$, we
must also have $v_5v_7\in \mv{T}$.  Since $v_6v_7,v_6v_8\notin \mv{T}$ and
$v_5v_7\in \mv{T}$, we get $d^+(v_2)=1$.  This also implies $d^+(v_{n-1})=1$.
Now when $n \ge 9$ our difference is counted by $-f_1(3)f_1(n-7)$. 
Here $f_1(3)$ accounts for the edges of the path from $v_2$ to $v_5$ and
$f_1(n-7)$ accounts for the edges of the path from $v_7$ to $v_{n-1}$ (and the
$-1$ accounts for the 9 edges that are present but not on either of these
paths).  Since $f_1(3)=1$, the total for this case is $-f_1(n-7)$.
When $n=8$ the total is $-f_1(3)=-1$ and when $n=7$ the total is 0, since
$v_{n-1}=v_6$.  Now by Lemma~\ref{path-lemma}, together with checking the cases
$n=7$ and $n=8$, we get that this case is counted by $-f_1(n-4)$.

Case (ii.2): 
${v_1w_1},{v_1v_2}\in \mv{T}$ and ${v_2w_1}\notin\mv{T}$ and also
${v_5w_5},{v_5v_6}\in \mv{T}$ and ${v_6w_5}\notin\mv{T}$.
Since ${v_1w_1},{v_1v_2}\in \mv{T}$, we must have ${w_1v_n}, {v_nv_1},
{v_{n-1}v_1}\in \mv{T}$ and ${v_1v_3} \not \in \mv{T}$.
Since ${v_5w_5},{v_5v_6}\in \mv{T}$, we must have ${w_5v_4}, {v_3v_5},
{v_4v_5}\in \mv{T}$ and ${v_5v_7} \not \in \mv{T}$.
Suppose $v_{n}v_2\notin \mv{T}$.  Now $v_2v_4\notin \mv{T}$, so $d^+(v_4)=1$.
Now our problem reduces to computing $-f_1(n-6)$; the $f(n-6)$ accounts for the
edges on the path from $v_6$ to $v_{n-1}$ and the $-1$ accounts for the 11 other
edges that are present.
Suppose instead that $v_{n}v_2\in \mv{T}$.  Now our problem reduces to computing
$f_2(n-4)$, accounting for the edges on the two paths from to $v_1$ (after
replacing $v_{n}v_2$ by $v_{v}v_1$) and the 12 edges present but not on these
paths.

So, combining the contributions from all cases we get that the difference is
$1  -f_1(n-4) - f_1(n-4) - f_1(n - 6) + f_2(n-4)$. By Lemma~\ref{path-lemma}
this is $1 - 2(f_1(n-4) + f_1(n-6)) \ne 0$ when $n \ge 8$.  When $n = 7$ the
difference is $1 -2f_1(3) - 1 + f_2(3) = -1$.
\end{proof}

For $n \ge 4$, a subgraph $\mv{T}\subseteq \mv{P_n}$ is \emph{extra weakly eulerian} if each vertex
$w\notin\{v_1,v_2,v_{n-1,}v_n\}$ satisfies $d^+(w)=d^-(w)$,
$d^+(v_1)=d^-(v_n)=1$, $d^+(v_2) = d^-(v_2) + 1$ and $d^-(v_{n-1}) = d^+(v_{n-1}) + 1$
Let $EE^*(\mv{P_n})$ (resp. $EO^*(\mv{P_n}$))
denote the set of even (resp. odd) extra weakly eulerian subgraphs.
Finally, let $g(n)=|EE^*(\mv{P_n})|-|EO^*(\mv{P_n})|$.
%
Lemma~\ref{path-lemma2} is analogous to Lemma~\ref{path-lemma}, but for extra
weakly eulerian subgraphs.

\begin{lemma}
If $n = 3k + j \ge 4$ for a positive integer $k$ and $j\in \{-1,0,1\}$, then $g(n) = -j$.
\label{path-lemma2}
\end{lemma}
\begin{proof}
Let $\mv{T}\subseteq \mv{P_n}$ be extra weakly eulerian.  Consider the directed paths $v_1v_3$ and $v_1v_2, v_2v_3$.  If $\mv{T}$ contains all of one path but none of the other, then we can
pair $\mv{T}$ with its complement which has opposite parity.  If neither of these cases holds, then we must have either $v_1v_3, v_2v_3 \in \mv{T}$ and $v_1v_2 \notin \mv{T}$ or $v_1v_2 \in \mv{T}$ and $v_1v_3, v_2v_3 \notin \mv{T}$.  The latter case is impossible, so suppose we have $v_1v_3, v_2v_3 \in \mv{T}$ and $v_1v_2 \notin \mv{T}$.  Then $v_3v_4, v_3v_5 \in \mv{T}$ and $v_2v_4 \notin \mv{T}$.  Hence the difference is counted by $g(n - 3)$.  It remains only to check that $g(4) = -1$, $g(5) = 1$ and $g(6) = 0$.
\end{proof}

\begin{figure}[ht]
\begin{center}
\begin{tikzpicture}[scale = 8]
\tikzstyle{VertexStyle}=[shape = circle,	
								 minimum size = 6pt,
								 inner sep = 1.2pt,
                                 draw]
\Vertex[x = 1.25, y = 0.699999988079071, L = \tiny {$v_1$}]{v0}
\Vertex[x = 1.04999995231628, y = 0.550000011920929, L = \tiny {$v_8$}]{v1}
\Vertex[x = 1.45000004768372, y = 0.550000011920929, L = \tiny {$v_2$}]{v2}
\Vertex[x = 1, y = 0.350000023841858, L = \tiny {$v_7$}]{v3}
\Vertex[x = 1.5, y = 0.350000023841858, L = \tiny {$v_3$}]{v4}
\Vertex[x = 1.04999995231628, y = 0.149999976158142, L = \tiny {$v_6$}]{v5}
\Vertex[x = 1.45000004768372, y = 0.149999976158142, L = \tiny {$v_4$}]{v6}
\Vertex[x = 1.20000004768372, y = 0.849999994039536, L = \tiny {$w_1$}]{v7}
\Vertex[x = 1.20000004768372, y = -0.100000023841858, L = \tiny {$w_5$}]{v8}
\Vertex[x = 1.25, y = 0.050000011920929, L = \tiny {$v_5$}]{v9}
\Edge[style = {post}](v0)(v2)
\Edge[style = {post}](v0)(v4)
\Edge[style = {post}](v0)(v7)
\Edge[style = {post}](v1)(v0)
\Edge[style = {post}](v1)(v2)
\Edge[style = {post}](v2)(v4)
\Edge[style = {post}](v2)(v6)
\Edge[style = {post}](v2)(v7)
\Edge[style = {post}](v3)(v0)
\Edge[style = {post}](v3)(v1)
\Edge[style = {post}](v4)(v6)
\Edge[style = {post}](v4)(v9)
\Edge[style = {post}](v5)(v1)
\Edge[style = {post}](v5)(v3)
\Edge[style = {post}](v5)(v8)
\Edge[style = {post}](v6)(v5)
\Edge[style = {post}](v6)(v9)
\Edge[style = {post}](v7)(v1)
\Edge[style = {pre}](v7)(v8)
\Edge[style = {post}](v8)(v6)
\Edge[style = {post}](v9)(v3)
\Edge[style = {post}](v9)(v5)
\Edge[style = {post}](v9)(v8)
\end{tikzpicture}
\end{center}
\caption{The orientation for Lemma \ref{cycle+2pendant+edge}.}
\label{fig:cycle+2pendant+edge}
\end{figure}
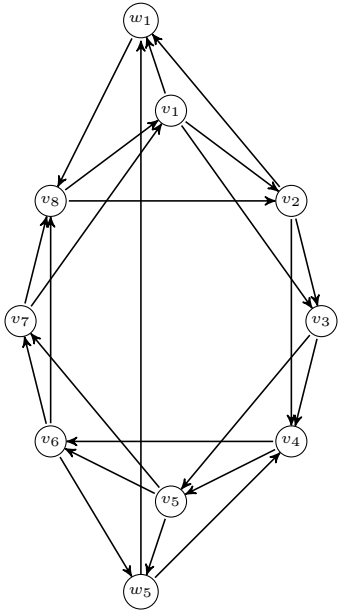

\begin{lemma}
\label{cycle+2pendant+edge}
8-cycle + two pendant edges + extra edge:
Let $J_8$ consist of an $8$-cycle on vertices $v_1,\ldots,v_8$ (in clockwise
order) with pendant edges at $v_1$ and $v_5$ leading to vertices $w_1$ and
$w_5$.  Form $\mv{D}_8$ by squaring $J_8$, adding the edge $w_1w_5$ and orienting the edges as follows. 
Orient edges
${v_iv_{i+1}}$ and
${v_iv_{i+2}}$ away from $v_i$ (with subscripts modulo $8$). 
Orient
${w_1v_8}$ away from $w_1$ and ${v_1w_1}$ and ${v_2w_1}$ toward $w_1$; 
similarly, orient
${w_5v_4}$ away from $w_5$ and ${v_5w_5}$ and ${v_6w_5}$ toward $w_5$. Finally, orient
${w_5w_1}$ toward $w_1$.
We will show that $f(\mv{D_8})\ne 0$ (or else $f(\mv{D_8}\setminus B)\ne 0$ for some subset $B\subseteq \{w_1,w_5\}$).
\end{lemma}

\begin{proof}
Form $\mv{D_8}$ as in the lemma.  Suppose $f(\mv{D_8}\setminus B) = 0$ for each subset $\emptyset \ne B\subseteq \{w_1,w_5\}$.  Then by Lemma~\ref{cycle+2pendant}, we have
$\diff(\mv{D_8} - w_5w_1) \ne 0$.  Hence it will suffice to show that the circulations of $\mv{D_8}$ containing $w_5w_1$ are half odd and half even.

Let $\mv{T}$ be a circulation of $\mv{D_8}$ containing $w_5w_1$.  Then ${w_1v_8} \in \mv{T}$ and ${v_1w_1},{v_2w_1}\notin\mv{T}$.  After suppressing $w_1$, we are looking at all circulations containing $w_5v_8$.

Consider the directed paths $v_5w_5$ and $v_5v_6, v_6w_5$.
If $\mv{T}$ contains all edges of one path and none of the other, then we can
pair $\mv{T}$ via a parity-reversing bijection.  So we assume we are not in one
of those cases.  Thus either (i) ${v_6w_5}\in \mv{T}$ and
${v_5w_5},{v_5v_6}\notin\mv{T}$, (ii) ${v_5w_5},{v_5v_6}\in \mv{T}$
and ${v_6w_5}\notin\mv{T}$, (iii) ${v_5w_5},{v_5v_6},{v_6w_5}\in \mv{T}$ or
(iv) ${v_6w_5},{v_5w_5} \in \mv{T}$ and ${v_5v_6} \notin \mv{T}$.

Case (i):
${v_6w_5}\in \mv{T}$ and ${v_5w_5},{v_5v_6}\notin\mv{T}$.
Then ${v_4v_6} \in \mv{T}$ and ${w_5v_4}, {v_6v_7}, {v_6v_8} \notin \mv{T}$. 
Now we can suppress $v_6$ and $w_5$. First suppose
${v_5v_7} \notin \mv{T}$.  Now $v_7, v_5 \notin \mv{T}$ and what remains is
counted by $-f_1(5)$.  Instead suppose ${v_5v_7} \in \mv{T}$. Then the
difference is counted by $g(7)$; the path is from $v_7$ to $v_5$.  Hence the
total difference is $g(7) - f_1(5) = -1 - (-1) = 0$.

Case (ii):
${v_5w_5},{v_5v_6}\in \mv{T}$ and ${v_6w_5}\notin\mv{T}$.
Then $v_3v_5, v_4v_5 \in \mv{T}$ and $w_5v_4, v_5v_7 \notin \mv{T}$.  Now we
can suppress $w_5$.  First suppose ${v_4v_6} \in \mv{T}$. 
There is only one possible circulation and it contains all edges except
$v_7v_8$; this circulation is odd, hence the difference is $-1$.  Now suppose
${v_4v_6} \notin \mv{T}$.  If $v_6v_7 \in \mv{T}$, then $v_6v_8 \notin \mv{T}$
and the difference is counted by $-g(6)$; the path is from $v_7$ to $v_4$.  If
$v_6v_7 \notin \mv{T}$, then $v_6v_8, v_8v_1, v_8v_2 \in \mv{T}$ and $v_7
\notin \mv{T}$.  Now the difference is counted by $-g(4)$; the path is from
$v_1$ to $v_4$.  Hence the total difference is $-1 - g(6) - g(4) = 0$.

Case (iii):
${v_5w_5},{v_5v_6},{v_6w_5}\in \mv{T}$.
Then $w_5v_4, v_3v_5, v_4v_5 \in \mv{T}$ and $v_5v_7 \notin \mv{T}$.  
If $v_4v_6, v_6v_7 \in \mv{T}$, then the difference is counted by
$g(6)$; the path is from $v_7$ to $v_4$.  Since $v_6v_7 \in \mv{T}$ and
$v_4v_6 \notin \mv{T}$ is impossible, we may assume either $v_4v_6 \in \mv{T}$
and $v_6v_7 \notin \mv{T}$ or $v_4v_6,
v_6v_7 \notin \mv{T}$.  Suppose we are in the former case.  Then $v_6v_8,
v_8v_1, v_8v_2 \in \mv{T}$ and $v_7 \notin \mv{T}$.  This difference is counted
by $g(4)$; the path is from $v_1$ to $v_4$.  Now suppose $v_4v_6, v_6v_7
\notin \mv{T}$.  Then $v_7 \notin \mv{T}$ and $v_6v_8 \notin \mv{T}$. 
This difference is counted by $f_1(4)$; the path is from $v_8$ to $v_3$. 
Hence the total difference is $g(6) + g(4) + f_1(4) = 0$.

Case (iv):
${v_6w_5},{v_5w_5} \in \mv{T}$ and ${v_5v_6} \notin \mv{T}$.
Then $w_5v_4, v_4v_6 \in \mv{T}$ and $v_6v_7, v_6v_8 \notin \mv{T}$.  
If $v_5v_7 \notin \mv{T}$, then $v_7 \notin \mv{T}$ and the
difference is counted by $f_1(6) = 0$; the path is from $v_8$ to $v_5$.  Hence
we may assume $v_5v_7 \in \mv{T}$.  Then $v_3v_5, v_4v_5 \in \mv{T}$ and the
difference is counted by $g(6) = 0$; the path is from $v_7$ to $v_4$.

So in each of the four cases, half the circulations are even and half are odd.
Thus, the difference is not affected by the circulations that use edge $w_5w_1$.
Now by Lemma~\ref{cycle+2pendant}, $f(\mv{D})\ne 0$, so $\mv{D}$ is
$d_1$-paintable.
\end{proof}

\section{Generalizing to Alon-Tarsi number}
\label{AT-section}
Excepting the direct proofs of paintability in Section \ref{DirectProofs}, we've actually proved that all the excluded subgraphs have a good Alon-Tarsi orientation.  This suggests that the main theorem might hold more generally for the Alon-Tarsi number $\AT(G)$---the least $k$ for which $G$ has an orientation $\vec{D}$ with $\Delta^+(\vec{D}) \le k-1$ and $EE(\vec{D}) \ne EO(\vec{D})$.   Here we show that this is indeed the case.

\begin{atmainthm}
If $G$ is a connected graph with maximum degree $\Delta\ge 3$ and
$G$ is not the Peterson graph, the Hoffman-Singleton graph, or a Moore graph
with $\Delta=57$, then $\AT(G^2)\le \Delta^2-1$.
\end{atmainthm}

The proof is identical to the paintability proof except we need to replace all
the auxiliary lemmas with their $\AT$ counterparts.  First the two subgraph
lemmas; these are actually easier to prove in the $\AT$ context.

\begin{lemma}
Let $G$ be a graph with maximum degree $\Delta$ and $H$ be an induced subgraph
of $G$ that is $d_1$-AT.  If $G\setminus H$ is $(\Delta-1)$-AT,
then $G$ is $(\Delta-1)$-AT.
\label{subgraphlemmaAT}
\end{lemma}

\begin{proof}
Let $G$ and $H$ satisfy the hypotheses.  Take an orientation of $G\setminus H$ demonstrating that it is $(\Delta-1)$-AT and an orientation of $H$ demonstrating that it is $d_1$-AT.
Now orient all the edges between $H$ and $G\setminus H$ into $G\setminus H$.
Call the resulting oriented graph $\vec{D}$. Then $\vec{D}$ satisfies the
outdegree requirements of being $(\Delta-1)$-AT since the outdegree of the
vertices in $G\setminus H$ haven't changed and the outdegree of each $v \in
V(H)$ has increased by $d_G(v) - d_H(v)$.  Since no directed cycle in $D$ has
vertices in both $H$ and $\vec{D}\setminus H$, the circulations of $\vec{D}$
are just all pairings of circulations of $H$ and $\vec{D}\setminus H$. 
Therefore $EE(\vec{D}) - EO(\vec{D}) = EE(H)EE(\vec{D}\setminus H) +
EO(H)EO(\vec{D}\setminus H) - (EE(H)EO(\vec{D}\setminus H) +
EO(H)EE(\vec{D}\setminus H)) = (EE(H) - EO(H))(EE(\vec{D}\setminus H) -
EO(\vec{D}\setminus H)) \ne 0$.  Hence $G$ is $(\Delta-1)$-AT.
\end{proof}

\begin{lemma}
Let $G$ be a graph with maximum degree $\Delta$ and let $H$ be an induced
subgraph of $G^2$.  If $H$ is $d_1$-AT, then $G^2$ is $d_1$-AT. 
If there exists $v$ with $d_{G^2}(v)<\Delta^2-1$, then $G^2$ is
$(\Delta^2-1)$-AT.
\label{subgraphlemmaG^2AT}
\end{lemma}

\begin{proof}
We prove the first statement first.  Form $G'$ from $G$ by contracting $V(H)$ to a single vertex $r$.  Let $T$ be
a spanning tree in $G'$ rooted at $r$.  Let $\sigma$ be an ordering of the
vertices of $G\setminus H$ by nonincreasing distance in $T$ from $r$.  Take an
orientation of $H$ demonstrating that it is $d_1$-AT; direct all edges between
$H$ and $G\setminus H$ towards $G\setminus H$ and direct all other edges
of $G^2$ toward the vertex that comes earlier in $\sigma$.  Call the resulting
oriented graph $\vec{D}$.  By construction, all circulations in $\vec{D}$ are
contained in $H$ and hence $EE(\vec{D}) \ne EO(\vec{D})$.  It is clear that
every vertex in $\vec{D}$ has indegree at least two and hence $G^2$ is
$d_1$-AT.

Now we prove the second statement, which has a similar proof.  
Suppose there exists $v$ with $d_{G^2}(v)< \Delta^2-1$.  As before we order the
vertices by nonincreasing distance in some spanning tree $T$ from $v$, and we
put $v$ and some neighbor $u$ last in $\sigma$.  Since $d_{G^2}(v)<\Delta^2-1$, either (i) $v$ lies on a 3-cycle
or 4-cycle or else (ii) $d_G(v)<\Delta$ or $v$ has some neighbor $u$
with $d_G(u)<\Delta$; in Case (ii), by symmetry we assume $d_G(v)<\Delta$.
In Case (i), $d_{G^2}(u)\le\Delta^2-1$ for some neighbor $u$ of $v$ on the short
cycle and by assumption $d_{G^2}(v)<\Delta^2-1$; so the two final vertices of
$\sigma$ are $u$ and $v$.  In Case (ii), we again have
$d_{G^2}(v)<\Delta^2-1$ and $d_{G^2}(u)\le \Delta^2-1$, so again $u$ and $v$
are last in $\sigma$.  
\end{proof}

The proof of Lemma~\ref{subgraphlemmaG^2AT} proves something slightly more
general, which we record in the following corollary.

\begin{cor}
Let $G$ be a graph with maximum degree $\Delta$ and let $H$ be an induced
subgraph of $G^2$.  Let $f(v)=d(v)-1$ for each high vertex of $G^2$ and
$f(v)=d(v)$ for each low vertex.  If $H$ is $f$-AT, 
then $G^2$ is $(\Delta^2-1)$-AT. 
\end{cor}

Now each of Lemmas~\ref{farlinked}, \ref{3unlinked}, \ref{B1B2},
\ref{path-lemma}, \ref{cycle+pendant}, and~\ref{cycle+2pendant} was already
proved for $AT$.  It
remains to prove the lemmas in Section \ref{DirectProofs} for $AT$.  We do this
by exhibiting in Figures \ref{indirect-fig1} and \ref{indirect-fig2} a good
Alon-Tarsi orientation for each.  For brevity, we will not
prove here that the counts differ; instead we give the actual even/odd
circulation counts for the reader to check at her leisure.  Each vertex will be
labeled with its indegree for easy checking.  Note that three of the cases in
Lemma~\ref{K4v2E2} are handled by Lemmas~\ref{K2vC4}, \ref{K3vP4},
and~\ref{K3vK1+P3} (none of which depend on Lemma~\ref{K4v2E2}).

We conclude by generalizing the conjectures we mentioned in the introduction to
the Alon-Tarsi number.

\begin{conj}[Borodin-Kostochka Conjecture (Alon-Tarsi version)]
If $G$ is a graph with $\Delta\ge 9$ and $\omega\le \Delta-1$, then $\AT(G)\le
\Delta-1$.
\end{conj}
\input{IndirectDirectProofs} 
\clearpage

\bibliographystyle{siam}
\bibliography{GraphColoring}
\end{document}